\documentclass{article}
\usepackage[utf8]{inputenc}

\usepackage{graphicx}
\usepackage{amsmath}
\usepackage{amssymb}
\usepackage{amsthm}
\usepackage{amsfonts}
\usepackage{tikz-cd}
\usepackage{mathabx}
\usepackage{tikz}
\usepackage{hyperref}
\usepackage[shortlabels]{enumitem}
\usepackage{dsfont}
\usepackage{wrapfig}
\usepackage{quiver}

\usepackage[strict]{changepage}

\DeclareMathAlphabet{\mathpzc}{OT1}{pzc}{m}{it}
\newcommand{\cat}[1]{\mathpzc{#1}}

\newcommand{\Ob}{\cat{Ob}}

\DeclareMathOperator{\Hom}{Hom}

\DeclareMathOperator{\Aut}{Aut}

\DeclareMathOperator{\Span}{Span}
\DeclareMathOperator{\CSpan}{\overline{Span}}

\DeclareMathOperator{\ev}{ev}
\DeclareMathOperator{\Rec}{Rec}

\newcommand{\isomto}{\xrightarrow{\sim}}

\newcommand{\GL}{\textrm{GL}}

\newcommand{\loc}{\textrm{loc}}

\newcommand{\Loc}{\textrm{Loc}}

\newcommand{\id}{\textrm{id}}

\DeclareMathOperator{\coim}{coim}
\DeclareMathOperator{\im}{im}
\DeclareMathOperator{\coker}{coker}

\DeclareMathOperator*{\esssup}{ess.sup}

\newcommand{\Q}{\mathbb{Q}}
\newcommand{\N}{\mathbb{N}}

\newcommand{\R}{\mathbb{R}}
\newcommand{\C}{\mathbb{C}}

\newcommand{\F}{\mathbb{F}}
\newcommand{\T}{\mathbb{T}}

\newcommand{\Ac}{\mathcal{A}}
\newcommand{\Bc}{\mathcal{B}}
\newcommand{\Cc}{\mathcal{C}}

\newcommand{\Fc}{\mathcal{F}}
\newcommand{\Gc}{\mathcal{G}}

\newcommand{\Lc}{\mathcal{L}}
\newcommand{\Mc}{\mathcal{M}}

\newcommand{\Pc}{\mathcal{P}}

\newcommand{\Sc}{\mathcal{S}}

\newcommand{\Abb}{\mathbb{A}}
\newcommand{\BB}{\mathbb{B}}

\newcommand{\EE}{\mathbb{E}}

\newcommand{\KK}{\mathbb{K}}

\title{A Categorical Framework for the Direct Integration of Banach Spaces}
\author{Daniel Funck, Giacomo Gavelli}
\date{}

\begin{document}
\newtheorem{theorem}{Theorem}[subsection]
\newtheorem{lemma}[theorem]{Lemma}
\newtheorem{claim}[theorem]{Claim}
\newtheorem{corollary}[theorem]{Corollary}
\newtheorem{prop}[theorem]{Proposition}
\newtheorem{conjecture}[theorem]{Conjecture}
\newtheorem*{theorem*}{Theorem}

\newtheorem{definition}[theorem]{Definition}
\newtheorem*{nameddefinition}{Definition of Ext}
\theoremstyle{remark}
\newtheorem{remark}{Remark}[subsection]
\newtheorem*{remarks}{Remarks}
\newtheorem*{warning}{Warning}
\newtheorem*{aside}{Aside}
\newtheorem{calc}[theorem]{Calculation}

\theoremstyle{remark}
\newtheorem{example}{Example}
\newtheorem{counterexample}[example]{Counterexample}
\newtheorem*{terminology}{Terminology}

\maketitle
\begin{abstract}
    We construct a categorical framework in which we can view the direct integral as a functor taking as input objects in a certain category of Abstract Banach bundles and outputting a Banach space. We do this by constructing a quasi-abelian category of abstract Banach bundles, and we prove certain foundational and theoretical results about this category and its objects. As a consequence, we are able to define a notion of direct integral of sheaves.
\end{abstract}
{\small\tableofcontents}

\section{Introduction}

Throughout mathematics, it is often useful to pass freely between collections of `small' objects and `large' universal objects, and representation theory is no exception. When $G$ is a finite group, the representations of $G$ correspond to the modules of the non-commutative ring $\C[G]$, and furthermore, there is a decomposition 
\[\C[G]\cong \bigoplus_{\pi}\pi^{\dim(\pi)}\]
so that all irreducible representations are found inside the regular representation $\C[G]$, considered as a left (equivalently right) module. 

When one turns one's attention to locally compact topological groups $G$ such as $\GL_n(\R)$ or $\GL_n(\Q_p)$, two important differences arise. Firstly, there are now typically uncountably many irreducible representations. Secondly, the group algebra $\C[G]$ becomes an unwieldy and unreasonable object due to its independence from the natural topology on $G$. Instead, one takes inspiration from harmonic and functional analysis, and studies representations such as that on $L^p(G)$. These representations play a similar role to the regular representation $\C[G]$ in the representation theory of finite groups. 

Introduced by von Neumann in the 1940s, the direct integral is a fundamental tool from the study of harmonic analysis and the representation theory of locally compact groups. It generalises direct sums in the same way integration generalises finite summation. Consequently, direct integrals provide a natural language for decomposing many of the analytic objects arising in representation theory.
A fundamental example is provided by the Plancherel decomposition of the left regular representation $L^2(G)$. If $G$ is a ``nice" locally compact group, then the Plancherel theorem may be interpreted as an isomorphism of unitary representations
$$L^2(G)\cong\int_{\widehat{G}}^\oplus \pi d\mu_{Pl},$$
where $\widehat G$ denotes the unitary dual of $G$, equipped with the Plancherel measure.

This direct integral is constructed as follows. Given a family of Hilbert spaces $(H_x)_{x\in M}$ over a measure space $(M,\mu)$:
\begin{enumerate}
    \item Choose a family of vectors $e_i(x)\in H_x$ indexed on $i\in I$ and $x\in M$ such that, for each $x\in M$, the closed span of $\{e_i(x)\}_{i\in I}$ is $H_x$;
    \item Declare a family $v_x\in H_x$ of vectors to be measurable if for all $i\in I$ the function $x\mapsto \langle v_x, e_i(x)\rangle$ is measurable;
    \item Define $\int_M^\oplus H_x d\mu$ as the space of all measurable families $v_x$ whose norm function $x\mapsto \|v_x\|_{H_x}$ lies in $ L^2(M,\mu)$.
\end{enumerate}
This construction of the direct integral relies heavily on the inner product structures on $H_x$ 
and therefore cannot apply to more general representations on Banach or Fréchet spaces. 
Admissible irreducible representations do not necessarily carry $G$-invariant Hilbert space structures, and sometimes the most natural topology on these spaces arises in other fashions. One nice example is the moderate growth topology defined in \cite{Cassel89}. We are therefore interested in generalising the direct integral to the more general setting of Banach spaces, to allow the study of a larger toolbox of examples such as $L^p(G)$.

The first question that this paper aims to answer: `given a family of Banach spaces $B_x$ indexed by some measure space $(M,\mu)$, can one define a sensible notion of direct integral $\int^{\oplus} B_xd\mu(x)$?'

In the above construction of Hilbert direct integrals, one can (and often does) choose the collection $e_i(x)$ so that the non-zero elements $\{e_i(x)\}_{i\in I}$ form an orthonormal basis of $H_x$. This orthonormal basis plays two roles in this construction. The first, a basis of $H_x$. The second, a functional $\langle \cdot,e_i(x)\rangle$ on $H_x$ to test $v_x$ for measurability. In the case of (Archimedean) Banach spaces, to define the construction, we replace this notion of orthonormal basis  with the notion of Markushevich bases. These consist of biorthonormal pairs $(b_i,b^*_i)_{i\in I}$ satisfying some additional properties, and work well for our purposes because every separable Banach space admits such a basis.

Changing gear for a moment, consider the following. Let $F$ be a local field, either Archimedean or non-Archimedean.
The local Langlands correspondence for $\GL_n$, proven by Langlands for Archimedean fields (see e.g. \cite{KnappLLC}) and by Harris-Taylor \cite{HT01} for $p$-adic fields, stipulates the existence of a bijective map 
\begin{equation*}
    \left\{
    \begin{array}{c}
        \textrm{admissible representations} \\ 
        \textrm{of }\GL_n/{F}
    \end{array}
    \right\}_{/\cong}\xrightarrow[]{\Rec} \left\{
    \begin{array}{c}
    \text{`semisimple' representations }\\
    \rho:W_F\to \GL_n(C)
    \end{array}
    \right\}_{/\GL_n-\textrm{conjugacy}}
\end{equation*}

Over the last 30 years, this correspondence, and its `global' analogue, have played a central role in the solutions of many problems, most famously the proof of Fermat's last theorem. There have also been efforts to give this correspondence additional structure, in the form of a functor
\[\Rec:\cat{Rep}^{\textrm{adm}}(G)\to \cat{Sh}(\Loc_{{}^LG})\] 
that assigns each admissible representations a certain kind of sheaf on some space of local parameters. 
This functorial approach gives us the opportunity to extend the Langlands correspondence to not only non-irreducible admissible representations, but also non-admissible representations such as $L^2(G)$ in the following manner:
\[\begin{tikzcd}
	{L^2(G)} && {??} \\
	{\int^{\oplus}_{\hat{G}}\pi}  && {\int^{\oplus}_{\hat{G}}\Rec(\pi)}
	\arrow[dotted, from=1-1, to=1-3]
	\arrow["{\textrm{disintegration}}"', from=1-1, to=2-1]
	\arrow["\Rec"', from=2-1, to=2-3]
	\arrow["{\textrm{integration}}"', from=2-3, to=1-3]
\end{tikzcd}\]

Of course, this schema begs the question: `What does it mean to take a direct integral of sheaves?'

The natural way to address this issue is to define the direct integral on both the objects and the restriction maps between them.
Our perspective throughout is therefore that direct integration should be regarded not merely as a construction on objects but as a functor. This requires a category whose morphisms encode measurable families of bounded operators and whose algebraic structure is rich enough for homological arguments. The principal achievements of this paper are the construction of two categories of measurable and bounded Banach families, $\cat{mBan}^M$ and $\cat{bBan}^M$,
and the proof that these are quasi-abelian.

We can then take the following intuitive approach to defining direct integrals on sheaves.
Suppose that $(\mathcal F_x)_{x\in M}$ is a measurable family of Banach-valued sheaves on a topological space $X$. For every open set $U\subseteq X$, the family of Banach spaces $\{\mathcal{F}_x(U)\}_{x\in M}$
may be integrated, producing a Banach space
$$\int_M^\oplus \mathcal{F}_x(U)d\mu$$
for each $U$ and, for each $V\subseteq U$, a restriction map $\int_M^\oplus \mathcal{F}_x(U)d\mu\to \int_M^\oplus \mathcal{F}_x(V)d\mu$, which together produce a presheaf.

This raises the question of whether the assignment
$$U\mapsto \int_M^\oplus \mathcal{F}_x(U)d \mu$$
again defines a sheaf.
A simple example illustrates why, in general, it does not. Let $X=M=\R$, and for each $x\in\R$ let $\mathcal F_x$ denote the skyscraper sheaf supported at $x$. Then
$$\int_\R^\oplus\mathcal{F}_x(U)dx\cong L^2(U),$$
so that the direct integral recovers one of the fundamental presheaves of functional analysis (gluing of $L^2$ functions may not be in $L^2$). Understanding the failure of the direct integral to produce a sheaf therefore amounts to understanding the categorical properties of bounded Banach families together with the exactness properties of direct integration. We address this question fully in Proposition \ref{prop: notsheaf} and the subsequent remark. This question is closely connected with representation theory. Locally constant sheaves play a central role in geometric representation theory and in categorical approaches to the Langlands programme. In a companion paper, the first author applies the theory developed here to establish a categorical Langlands correspondence for separable (not necessarily admissible) unitary representations.

The main goals of this paper are twofold. First, the construction of categories $\cat{mBan}^M$ and $\cat{bBan}^M$ whose objects consist of measurable families of Banach spaces, which one can dub `Abstract Banach Bundles' and whose morphisms are measurable (resp. measurable and uniformly bounded) families of morphisms between them. Second, the construction of `exact'\footnote{To be defined more precisely in Section 3} functors $\int^{\oplus,p}_M:\cat{bBan}^M\to \cat{Ban}$ which generalise the notion of direct integral. After constructing these categories, we prove some results about these categories. Among these is the very important result:

\begin{theorem*}({Theorems \ref{Thm: mBan QA} and \ref{thm: bBan QA}})
    The categories $\cat{mBan}^M$ and $\cat{bBan}^M$ are quasi-abelian.
\end{theorem*}
In the literature there are a number of different notions of Banach bundles, and as far as the authors are aware, none of them have this property. In Section \ref{Section: Banach bundle}, we give more details explaining why some of the alternative contenders are insufficient for our purposes.

\subsection*{Order of the paper}

In Section 2, we introduce a notion of measurable families of Banach spaces indexed on a Polish space $M$ with a Borel measure $\mu$, as well as a notion of measurable families of morphisms between such objects, forming the category $\cat{mBan}^M$. We then prove some elementary measurability results and prove that $\cat{mBan}^M$, and the closely related $\cat{bBan}^M$ are quasi-abelian (Theorems \ref{Thm: mBan QA} and \ref{thm: bBan QA}).
Finally, we end the section by mentioning some related categories. 

In Section 3, we introduce the measurable sections functor as a functor $\int^m_M:\cat{mBan}^M\to \cat{Vect}$, construct the $p$-direct integral as a functor from a certain subcategory $\int^{\oplus,p}_M:\cat{bBan}^M\to \cat{Ban}$ whose morphisms are uniformly bounded, and show that these functors are strictly and strongly exact, respectively.

The results in Section 4 apply largely only in the case where the base field $\KK$ is Archimedean. This is due to the fact that we often require stronger properties of the direct integral or the Markushevich bases involved than those afforded by non-Archimedean Banach spaces. 
Despite the loss of generality, this section contains (in the humble opinion of the first author) the most interesting theorems of the paper In particular, we show: 
\begin{theorem*}
    Let $(B_x)$, $(C_x)$ be two objects of $\cat{bBan}^M$, and $(T_x):(B_x)\to (C_x)$ a morphism.
    \begin{description}
        \item[(Theorem \ref{thm: approximate isometries} and Corollary \ref{Cor: approximate isometries})] If $U\subseteq M$ is a small neighbourhood of $x\in M$, then one can approximately-isometrically embed the Banach space $B_x$ into the direct integral $\int^{\oplus,p}_U B_tdt$. Furthermore, this approximation improves arbitrarily close to an isometry as one shrinks the neighbourhood $U$.
        In particular, one can reconstruct the constituent $B_t$ from $\int^{\oplus,p}_M B_tdt$.
        \item[(Proposition \ref{Prop: Direct integral duals})] Suppose $p\in(1,\infty)$ and $q=\frac{p}{p-1}$. Under certain additional hypotheses, the dual space of 
        $\int^{\oplus,p}_M B_tdt$ is $\int^{\oplus,q}_M B_t^{\vee}dt$.
        \item[(Theorem \ref{Thm: integral kernels})] Under technical hypotheses on $p,q\in [1,\infty)$ and on the families $(B_x),(C_x)$, any bounded operator 
        \[T:\int^{\oplus,p}_MB_mdm\to \int^{\oplus,q}_NC_ndn\]
        can be decomposed as an integral kernel
        \[T(v_m)=\int_{} k_T(m,n) v_m  d\mu_n(m)\]
        where each $k_T(m,n)$ is a bounded operator from $B_m\to C_n$. 
    \end{description}
\end{theorem*}

In Section 5, we turn our attention to measurable and bounded families of sheaves. We demonstrate some origins of measurable families of constructible sheaves and show some interesting examples of their direct integral.

This paper also contains two appendices. The first contains a technical result that can be thought of as a generalisation of the fact that a function $f:M\to B$ from a measure space to a \emph{separable} Banach space is strongly measurable if and only if it is weakly measurable.

The second appendix is a list of references for non-Archimedean analogues of certain algebraic results from classical functional analysis, such as the Hahn-Banach theorem. This just simply for completeness, so that we can reference these standard results and not worry whether they apply in the non-Archimedean case.

\subsection{Relation to Banach bundles}{\label{Section: Banach bundle}}

In the literature, there are multiple notions of `Banach bundles', to which our category $\cat{bBan}^M$ seems to have similarities to. However, it appears to be genuinely different to all of them, for technical reasons. 

Here are two examples of notions of Banach bundles we are aware of. 

\begin{description}
    \item[First Contender (See Chapter 3 of \cite{MR1335233}):] This is the notion of vector bundle found in Lang's book \cite{MR1335233}. A surjective continuous map $\pi:E\to X$ of topological spaces such that each $\pi^{-1}(x)$ is a Banach space, and such that for each point $x\in X$, there is a neighbourhood $U_i$ and a Banach space $X_i$ such that $\pi^{-1}(U_i)\cong U_i\times X_i$. 

    This is different to what we attempt to do, because we allow the Banach spaces $B_x$ lying above each $x$ to vary.
    \item[Second Contender (See \cite{MR4678838}):] Given a fixed Banach space $\BB$, one calls a weakly measurable multivalued map $\mathbb{E}: X\to \BB$ a Banach $\BB$-bundle on $X$ if $\mathbb{E}(x)$ is a closed linear subspace of $\BB$ for every $x\in X$. 

    The problem with this approach is that, even in the situation with $X=*$, the cokernel of a inclusion $S\hookrightarrow B$ of a closed subspace may well not embed into $\BB$, because a closed subspace need not be complemented. 
    Indeed, a Theorem of Lindenstrauss and Tzafriri (Theorem 2.d.9 of \cite{MR500056}) implies that any $\ell^p$-space with $p>2$ has subspaces which cannot be complemented.

    What's more, even if one were to choose $\BB=C([0,1])$ so that one can guarantee that every separable Banach space embeds isometrically, one has very little control over how each $\coker(\EE_1(x)\to \mathbb{E}_2(x))$ embeds. 
    So unless one is willing to seriously restrict one's universe, this notion is also a dead-end.  
\end{description}
Our approach is to avoid embedding the families $(B_x)$ into any Banach space at all, and to instead look at particular minimal generating sequences known as Markushevich bases. Even so, this runs into difficulties, because given a pair of Banach spaces $B,C$ with respective M-bases $\{(b_i,b_i^*)\}_{i\in \N}$ and $\{(c_i,c_i^*)\}_{i\in \N}$, and some Bounded morphism $T:B\to C$, there is no way to algorithmically construct an $M$ basis of neither $\ker(T)$ nor $\coker(T)$ from those of $B$ and $C$, due mainly to the fundamental difficulty that there is no guarantee of a projection $P$ from $B$ onto $\ker(T)$, nor an inclusion $\coker(T)\to C$.

\subsection*{Acknowledgements}

This paper was produced as part of the first author's time as a Teach@T\"ubingen Fellow and the second authors time as a PhD candidate. The authors would like to thank the relevant funding bodies, the Teach@T\"ubingen Fellowship stipend, and the University of T\"ubingen for financially supporting the authors and bringing together the expertise to make this research possible.
We would also like to offer particular thanks to Anton Deitmar for his unwavering support and encouragement throughout this project. 

\section{The categories of measurable and bounded abstract Banach bundles}

\subsection{The category of measurable abstract Banach bundles}

Suppose we are given the following data: $\KK$ a separable, locally compact, spherically-complete normed field with a Borel measure (we really think $\KK=\R,\C$, $\Q_p$, $\F_p((t))$ or similar). 
Throughout, we let $\cat{Vect}$ be the abelian category of $\KK$-vector spaces, $\cat{Ban}$ be the category of $\KK$-Banach spaces whose morphisms are precisely  and, in the case $\KK=\C$, $\cat{Hilb}$ be the category of Hilbert spaces. We also denote by $\omega$ a limit ordinal. 

We first recall the notion of a bounded Markushevich basis of a Banach space. A detailed reference can be found at \cite{Tse25}.

\begin{definition}{\label{Def: Markushevich Basis}}
    Suppose $B$ is a Banach space and $B^\vee$ its continuous dual. A \emph{Markushevich basis} (M-basis) of $B$ is a collection of pairs $\{(b_i,b_i^*)\in B\times B^\vee\}_{i\in \omega}$ such that:
    \begin{enumerate}
        \item The set $\{b_i\}_{i\in \omega} $ is \emph{complete}; i.e. the closed span $\CSpan(b_i)=B$;
        \item The set $\{b_i\}$ is \emph{minimal}, meaning for any $i\in \omega$, the closed span $\CSpan\{b_j:j\neq i\}\subsetneq B$;
        \item The collection $(b_i,b_i^*)$ is \emph{biorthonormal}, meaning that $b_k^*(b_i)=\delta_{ij}$ is the Kronecker delta function;
        \item The set $\{b_i^*\}$ is \emph{total}, or \emph{separates the points} of $B$; meaning that if $b_i^*(x)=0$ for all $i\in \omega$, then $x=0$.
    \end{enumerate}
        If, in addition, there exists some $C\geq1$ such that for every $i\in\omega$, 
        \[\|b_i\|\|b_i^*\|\leq C,\]
        then we say that the M-basis is \emph{bounded}.
\end{definition}

\begin{remark}
    It is worth pointing out that, due to minimality of the set $\{b_i\}$, the family $\{b_j^*\}$ of biorthonormal functions is uniquely determined from the family $\{b_i\}$, because each $b_j^*\in B^\vee$ is uniquely determined, via continuity, by its values on the dense subset $\Span\{b_i\}$, in turn uniquely determined by the property $b_j^*b_i=\delta_{ij}$.
\end{remark}

In this paper, we will further take the convention that an M-basis is normalised (i.e. that $\|b_i\|=1$ for all $i\in \omega$). We remark that a bounded M-basis with constant $C=1$ is called an Auerbach basis, and it is conjectural whether or not every separable Banach space has one. Despite this open problem, it is known (even so early as 1943) that every separable Banach space has an M-basis\footnote{In fact, given a constant $C>1$, Pełczyński shows in \cite{MR425587} the existence of a bounded M-basis with constant $C$.}, which one can obtain using a process akin to the Gram--Schmidt algorithm. We mention this, because the process will prove useful later in the paper.  
It is usually the convention to only give the name Markushevich/Auerbach basis for a Banach space if it is countable. We don't strictly use this convention, though the majority of the paper is only concerned with spaces that are indeed separable.

We are now ready to define the category of measurable abstract Banach bundles $\cat{mBan}^M_{\omega}$ of $\KK$-Banach spaces as follows.

\begin{definition}[\textbf{Abstract Banach Bundles as objects of $\cat{mBan}^M_\omega$}]{\label{Def: Objects}}
Let $M$ be a topological space with a measure $\mu$ on the Borel $\sigma$-algebra, which we will refer to as a Borel measure space. Let $\omega$ be an ordinal. 
    A (measurable) \emph{abstract Banach bundle} in $\cat{mBan}^M_\omega$ is a family of tuples
    \[\Big\{\big(B_x,\{b_i(x)\}_{i\in\omega}\big)\Big\}_{x\in M}\]
    such that 
    \begin{enumerate}
        \item For each $x$, $B_x$ is a $\KK$-Banach space. 
        \item For each $x\in M$, the set $\{b_i(x)\}_{i\in \omega}$ is a sequence of vectors, either zero or norm $1$, such that the non-zero elements form an M-basis.
        \item For each $i\in \omega$, the zero set
        \[Z_i:=\{x\in M:b_i(x)=0\}\]
        is a measurable set. 
        \item{\label{Condition 4}} For any $(a_i)_{i\in \omega}\in k$, such that all but finitely many $a_i$ are zero, the expression
        \[\left\|\sum_ia_ib_i(x)\right\|_{B_x}\]
        is a measurable function $M\to \R$. 
    \end{enumerate}
\end{definition}

Our strategy to defining the notion of measurable morphisms of abstract Banach bundles is standard, first defining the notion of constant and step families, and then extending to pointwise limits. Of course, we want all of these notions to be stable under obvious operations such as linear spans and composition, so for the sake of these desired properties, it is convenient to add some restriction to the notions of constant and step families to simplify these proofs. As proposition \ref{Prop: Decent=Everything} will show, this restriction is minor.

\begin{definition}
    Suppose $(B,\{b_i\}_{i\in \omega})$ and $(C,\{c_i\}_{i\in \omega})$ are two Banach spaces with M-bases. We say that a bounded operator $T:B\to C$ is \emph{basic}\footnote{This notion appears in \cite{PE73} under the guise of `finite expansion operators'.} if it maps finite span vectors to finite spans. That is, 
    \[T(\Span \{b_i\})\subseteq \Span\{c_i\}.\]
    We denote the set of basic operators as $\operatorname{Basic}(B,C)$. 
    If $T\in \overline{\operatorname{Basic}}(B,C)\subseteq \Bc(B,C)$ is in the norm closure of the set of basic operators, we say that $T$ is \emph{decent}.
\end{definition}

It is worth pointing out some elementary facts about basic operators. Firstly, the identity operator in $\Bc(B,B)$ is basic. Secondly, by Lemma 1 of \cite{PE73}, all finite rank operators are decent. This proof also works in the non-Archimedean case, and we can even generalise thus:

\begin{prop}{\label{Prop: Decent=Everything}}
    Assume both $B$ and $C$ are separable Banach spaces, with M bases $\{b_i\}_{i\in \N}$ and $\{c_i\}_{i\in \N}$.
    Then the set of basic morphisms is norm-dense in the space of all bounded operators. 
    \[\overline{\operatorname{Basic}}(B,C)=\Bc(B,C).\]
\end{prop}
\begin{proof}
Let $T:B\to C$ be a bounded operator.
  For each $i$, there is a sequence $v_{i,n}:=\sum_{j=1}^n\alpha^{i,n}_jc_j$
    such that $\lim_{n\to \infty}v_{i,n} =Tb_i $.
    Let $M_i:=\|b_i^*\|$, and let $\epsilon>0$. 
    One can find a sequence $\epsilon_i>0$ of positive real numbers such that $\epsilon=\sum_i M_i\epsilon_i<\infty$. 
    (For example, one can set $\epsilon_i:= \frac{\epsilon}{2^iM_i}$).

For every $i$, there is some $N_i$ large enough so that $\|v_{i,N_i}-Tb_i\|<\epsilon_i$. Let $T^\epsilon$ be the linear map (defined on $\Span\{b_i\}$) mapping each $b_i\mapsto v_{i,N_i}$ so that $c^*_jT^\epsilon b_i=0$ for all $j>N_i$.
    Let $b=\sum_k a_kb_k$ be a finite sum with $\|b\|_B\leq 1$. Then

     \begin{align*}
        \|(T-T^\epsilon)b\|
        &= \left\|\sum_k a_k(T-T^\epsilon)b_k \right\| \\
        &\leq \sum_k |a_k|\|(T-T^\epsilon)b_k\| \leq \sum_k M_k\epsilon_k=\epsilon.
    \end{align*}
    Thus, since the set of such $b$ is dense in the unit ball of $B$, it follows firstly that $\T^\epsilon$ is bounded, and hence a basic operator, and further that $\|T-T^\epsilon\|<\epsilon$, and therefore one can always find some basic operator $T^\epsilon$ arbitrarily close to $T$. The result follows. 
\end{proof}

We can now finish defining the category $\cat{mBan}^M_{\omega}$. 

\begin{definition}[\textbf{Morphism in the category $\cat{mBan}^M_\omega$}]\label{Def: morphisms}
    To define morphisms, we need some intermediate notions. Let $(B_x,\{b_i(x)\}), (C_x,\{c_i(x)\})$ be two abstract Banach bundles, and suppose $(T_x:B_x\to C_x)_{x\in M}$ is a family of bounded operators. We say that:
    \begin{itemize}
        \item $(T_x)_{x\in M}$ is a \emph{constant basic morphism}, or a \emph{constant basic family}, if each $T(x)$ is basic, and the functions $T_{ij}:x\mapsto c_j^*(x)T_x b_i(x)$ are constant functions $M\to k$;
        \item $(T_x)_{x\in M}$ is a \emph{step basic morphism}, or a \emph{step basic family}, if $M$ can be decomposed into a countable disjoint union $M=\coprod_{i=1}^{\infty}M_i$ of measurable subsets such that each restriction $T|_{M_i}$ is constant basic. Equivalently, each of the functions $T_{ij}:M\to k$ above is a measurable function with countable image. 
        \item $(T_x)_{x\in M}$ is a \emph{Bochner measurable morphism}, or a \emph{measurable family}, if there is a sequence $S_n(x)$ of step basic families such that for almost all $x\in M$ we have
        \[\lim_{n\to \infty }S_n(x) =T(x)\]
        in the $x$-norm operator topology.
    \end{itemize}
    We define the set of morphisms $(B_x,\{b_i(x)\})\to(C_x,\{c_i(x)\})$ as the set of all Bochner measurable families. 
\end{definition}

We call the functions $T_{ij}:M\to k$ the matrix coefficients of $(T_x)$. We remark that all the definitions regarding morphisms are with respect to these functions, and not with respect to their norms. Indeed, a constant basic family may not have a constant norm!

For convenience, because the notation is somewhat cumbersome, we will often suppress the Markushevich basis in the notation when we talk about abstract Banach bundles. This is reasonable, since the main purpose of the basis is to enable us to allow us to deform the norm as $x\in M$ varies. 

Our choice of using M-bases in this definition is manyfold. Firstly, because we fundamentally need a way to compare vectors across different Banach spaces, we need some frame of reference to do so - one can imagine that the bases $\{b_i(x)\}$ are like rods fixing the families of Banach spaces to each other. This, of course, is identical to the notion used to compare families of Hilbert spaces $(H_x)$ in the classical Hilbert direct integral, except here there is no notion of orthogonality. 

Secondly, we choose M-bases because one is guaranteed, via the Gram--Schmidt algorithm (See Theorem 4.59 of \cite{FHH+11}), that \emph{every} $\R/\C$-Banach space has one. As a first attempt at defining this category, we tried using Schauder bases. One might have thought this would be ok, one simply restricts ones attention to spaces that have such a basis, but this approach hits the rather unpleasant issue that given a bounded operator $T:(B,\{b_i\})\to (C,\{c_i\})$, there is no way to translate either Schauder basis (or even the combination) to a Schauder basis of the kernel, for the reason that the kernel may not even have one. Constructed examples of Banach spaces without a Schauder basis such as that in \cite{PE73} are almost always constructed as closed subspaces of nice Banach spaces with Schauder bases, and there are even examples of such Banach spaces that are closed subspaces of $l^p(\N)$ for $p>2$ (see. e.g. \cite{MR508257}).

The non-Archimedean case is simpler. Every separable $\KK$-Banach space has a Schauder basis (see 5.5, 3.16 of \cite{VanR78}), and every Schauder basis is itself an M-basis. 

Finally, in an ideal world, it would have been nice to require that the basis in the data was in fact an Auerbach basis, that is, the biorthogonal functions $\|b_i^*\|=1$ are all norm $1$, better reflecting the construction of direct integrals in Hilbert spaces, since in Hilbert spaces, Auerbach and orthonormal bases are equivalent concepts. However, as the existence of such a basis is at best conjectural, and we don't require that the M-basis is even bounded for any proof to work, we stick to the more general setting. 

If $(B_x,\{b_i(x)\}_{i\in \omega})$ is an abstract Banach bundle such that each zero set $Z_i=\emptyset$ or $M$, and such that all functions in condition 4. are continuous (resp. constant), then we can call $(B_x)$ a continuous (resp. constant) abstract Banach bundle. 

Now, of course, we need to confirm that $\cat{mBan}^M$ is indeed a category, an exercise that is a trivial corollary of the following lemma:

\begin{lemma}{\label{lem: composition}}
    The composition of two (Bochner) measurable morphisms is a measurable morphism. 
\end{lemma}

\begin{proof}
    Let $(T_x):(B_x)\to (C_x)$ and $(S_x):(C_x)\to (D_x)$ be two morphisms measurable families between abstract Banach bundles.
    We start by showing that if both are constant basic, then their composition is also constant. Since it is clear that the compostition of basic operators is basic, we just want to show that each $x\mapsto d_j^*(x)S_xT_xb_i(x)$ is a constant function. 

    Because $S_x$ is basic, the functional $d_j^*(x)S_x:C_x\to k$ is equal to the finite span 
    \[\sum_s \alpha_sc_s^*(x)=d_j^*S_x\]
    where each $\alpha_s:=d_j^*(x)S_xc_s(x)$ is constant.
    It follows that the composition 
    
    \[\sum_s\alpha_sc_s^*(x)T_xb_i(x)=d_j^*S_xT_xb_i\]
    is constant. 
    
    Now it is simple to show that the composition of two step functions is also a step function, one takes the countable collection of sets $M_n^T$ and $M_n^S$ on which the functions $T_x$ and $S_x$ are constant respectively, and reorders the countable set $\{M_n^T\cap M_m^S:m,n\in \N\times\N\}$ to obtain a step function. 

    Lastly, one needs to show that the composition of two measurable families is a measurable family. If $(T_x)$ and $(S_x)$ are measurable families with respective families of step functions $(T_x^n)$ and $(S^n_x)$ which converge to them pointwise in the $x$-norm operator topology, then it is clear that the composition $S_x^n\circ T_x^n\to S_x\circ T_x$ converges in the $x$-norm operator topology too.
\end{proof}

\begin{lemma}\label{matrix coefficients are measurable maps}
    If $(T_x):(B_x)\to (C_x)$ is a morphism in $\cat{mBan}^M_\omega$, then for any $i,j\in\omega$ the matrix coefficient 
    \begin{equation}\label{eq: weakly measurable morphism}
    \begin{split}
        T_{ij}:M&\to k \\
        x&\mapsto c_j^*(x)\Big(T_x b_i(x)\Big)
        \end{split}
    \end{equation}
    is measurable.
\end{lemma}

\begin{proof}
    By assumption, there is a sequence of step basic families $(S^n_{x}):(B_x)\to (C_x)$, such that for almost every $x\in M$
    $$\sup_{\substack{v\in B_x\\ ||v||_{B_x}\le 1}}\Big|\Big| S^n_{x}(v)-T_x(v)\Big|\Big|_{C_x}\xrightarrow[]{n\to\infty}0.$$
    In particular, for $i\in\omega$
    $$\Big|\Big|S^n_{x}(b_i(x))-T_x(b_i(x))\Big|\Big|_{C_x}\xrightarrow[]{n\to\infty}0.$$
    Moreover, for $j\in\omega$, the map $c_j^*(x)$ is continuous with respect to the operator norm, so that 
    $$c_j^*(x)\Big(S^n_{x}\big(b_i(x)\big)-T_x\big(b_i(x)\big)\Big)\xrightarrow[]{n\to\infty}0.$$
    By linearity, this implies that the matrix coefficients $S^n_{ij}(x)$ converge pointwise (almost everywhere) to $T_{ij}(x)$, making $T_{ij}(x)$ measurable.
\end{proof}

\begin{definition}
    We call a collection of bounded operators $(T_x)_{x\in M}$ such that the matrix coefficients $T_{ij}(x)$ are all measurable a \emph{weakly measurable morphism}.
\end{definition}

There is a nice way to think about weakly measurable morphisms $\Hom_{\cat{mBan}^M_{\omega}}((B_x),(C_x))$. For each $x\in M$, the M-basis $(b_i(x))$ for $B_x$ gives continuous bounded maps
\begin{align}
    \ell^1(\omega)&\to B_x, \qquad B_x \to \ell^\infty(\omega) \\
    \\
    e_i&\mapsto b_i(x), \qquad v\mapsto (b^*_i(x)v)_{i\in \omega}
\end{align}

Thus any map $T_x:B_x\to C_x$ defines a bounded map $\ell^1(\omega)\to \ell^\infty(\omega)$, and it is a well known fact that $\Bc(\ell^1,\ell^\infty)\cong L^{\infty}(\omega\times \omega)$. Thus, $T_x$ can be viewed as an infinite matrix in $L^\infty(\omega\times\omega)$ with entries equal to $T_{ij}(x):=c^*_j(x)T_xb_i(x)$, and a family $(T_x)$ is weakly measurable if all matrix entries are.

\begin{remark}
   Weak measurability is weaker than Bochner measurability. The following is a standard counterexample:

    Let $(a_1,a_2,\cdots)$ be an enumeration of the countable set $\Q\cap[0,1]$, and define $A_t=\{n\in \N:a_n\leq t\}$, an uncountable chain in $\N$. Then 
    \[t\mapsto \mathds{1}_{A_t}\in \ell^{\infty}(\N)\cong\Hom_{\cat{mBan^{[0,1]}}}((\ell^1),(k))\]
    defines a map from $[0,1]$ to the set of homomorphisms from the constant abstract Banach bundle $(\ell^1)$ to the constant abstract Banach bundle $(k)$. The image is clearly uncountable and, because $\|\mathds{1}_{A_t}-\mathds{1}_{A_s}\|=1$ whenever $s\neq t$, has no countable dense subset. In the constant case, Pettis' measurability theorem applies, and we obtain that this map cannot be Bochner measurable, although all coefficients are simply indicator functions. 
\end{remark}

\subsection{Elementary measurability results}

The point of this section is to collect a series of results concerning measurability in the case when $\omega=\N$. 

We start with some notation. Let $\KK^{<\omega}$ be the set of functions $\omega \to k$ with finite support. Let $Q\subseteq k$ be a countable dense subset (in characteristic $0$, we can always take $Q$ as an algebraic extension of $\Q$), and let $Q^{<\omega}$ be the subset of $\KK^{<\omega}$ with values in $Q$. For an abstract Banach bundle $(B_x,\{b_i(x)\}_{i\in \omega})_{x\in M}$ in $\cat{mBan}^M$, define further
\begin{align*}
    S_k&:=\{\sum_i a_ib_i(x): (a_i)\in k^{<\omega}\} \\
    S_k^\vee&:=\{\sum_i a_ib_i^*(x): (a_i)\in k^{<\omega}\} \\
    S_Q&:=\{\sum_i a_ib_i(x): (a_i)\in Q^{<\omega}\} \\
    S_Q^\vee&:=\{\sum_i a_ib_i^*(x): (a_i)\in Q^{<\omega}\}.
\end{align*}
to be the sets of constant sections of $\Hom((k),(B_x)_{x\in M})$. 
Note that for each $x$, the sets $S_k$ and $S_Q$ can be viewed as dense subsets of $B_x$, and the sets $S_k^\vee$ and $S_Q^\vee$ as weak-$^*$ dense in $B_x^\vee$.  

\begin{terminology}
    In the following lemmas, when we refer to constant morphisms, we will not impose the condition of being basic on these morphisms/families. 
    We also refer to notions of constant/step/measurable/weakly measurable families of vectors $(v_x)\in (B_x)$, by which we mean a constant/step/measurable/weakly measurable (not necessarily basic) morphism from the constant family $(k)\to (B_x)$ respectively. 
\end{terminology}

\begin{lemma}
    Let $(T_x):(B_x)\to (C_x)$ be a constant (not-basic) family. Then $\|T_x\|_x$ is a measurable function. 

    Consequently, if $(v_x)\in (B_x)$ is a constant family of vectors, or $(\phi_x)\in (B_x^\vee)$ is a constant family, then both functions of $x$ given respectively by  $\|v_x\|_{B_x}$, $\|\phi_x\|_{B_x^\vee}$ are measurable. 
\end{lemma}

\begin{proof}
   A well known consequence of the Hahn-Banach theorem is that the norm of $T_x$ can be expressed as
\[\|T_x\|_x=\sup_{\substack{\|\phi\|\leq 1 \\ \|v\|\leq 1}}|\phi T_x v|\]
where $\phi \in C_x^\vee$ and $v\in B_x$. 

    A consequence of the Banach–Alaoglu Theorem is that the subset $\{\phi\in S_Q^\vee:\|\phi\|\leq 1\}$
    is a weak-$^*$ dense subset of the unit ball of $C_x^\vee$. It follows that the above supremum can be replaced with the countable supremum over the countable subsets $\{\phi\in S_Q^\vee:\|\phi\|\leq 1\}$ and $\{v\in S_Q:\|v\|\leq 1\}$.
    Since all functions of the form $\phi T_x v$ with $\phi\in S_Q^\vee$ and $v\in S_Q$ are measurable by Lemma \ref{matrix coefficients are measurable maps}, it follows that the supremum over countably many measurable functions is also measurable.

    The next statement comes by setting the source and target abstract Banach bundles as the constant abstract Banach bundle $(k,\{1\})_{x\in M}$ respectively. 
\end{proof}

Incidently, while the property of measurability for the functions 
\[x\mapsto \left\|\sum_k\beta_kb_i(x)\right\|_x, \qquad \beta\in k^{< \omega}\]
is part of the necessary data, the same statement for the set of finite spans of the biorthonormal functions 
$\{x\mapsto \|\sum_k\beta_kb^*_i(x)\|_x: \beta\in k^{< \omega}\}$ is a Theorem.

\begin{lemma}
    Let $(v_x)\in (B_x)$ be a measurable family. Let $n\in \N$ and $B^n_x$ be the span of $b_1(x),\cdots ,b_n(x)$. Then the function of $x$
    \[\operatorname{Dist}(v_x,B_x^n):=\inf_{w_x\in B_x^n} \|v_x-w_x\|_x\]
    is measurable.
\end{lemma}
\begin{proof}
    Extending to (Bochner) measurable families from constant families is easy, so we consider only this restricted situation.  
    By the previous lemma, each function $\|v_x-w_x\|_x$ is a measurable function of $x$.
    Since the above infimum can be replaced with the infimum over the countable dense subset $S_Q$, it follows that the distance is measurable.
\end{proof}

\begin{lemma}{\label{Lem: Rapid convergence}}
    Let $v_x\in (B_x)$ be a measurable section. Then there a sequence of step functions $f_n(x)$ with each $f_n(x)\in B_x^n$ such that for almost all $x\in M$,  $f_n(x)\to v_x$ as $n\to \infty$. 
\end{lemma}

\begin{proof}
    Let $d_n(x):=\operatorname{Dist}(v_x,B_x^n)$, measurable by the previous lemma. Because the sequence $\{b_i(x)\}$ has a dense span in $B_x$, it is clear that $\lim_{n\to \infty}d_n(x)=0$ for every $x\in M$. 

    Let $(\epsilon_n)$ be a sequence of positive real numbers that converge to $0$.
    As $S_Q\cap B_x^n$ is countable and dense in $B_x^n$, we can find some function $f_n:M\to S_Q$ with countable image such that $\|v_x-f_n(x)\|_x<d_n(x)+\epsilon_n$. 
    We therefore see that $f_n(x)$ satisfies the desired conditions. 
\end{proof}

\begin{lemma}[Gram--Schmidt]{\label{Lem: Gram--Schmidt}}
    Let $(B_x)$ be a collection of Banach spaces. Let $\{v_j(x)\}_{j\in \omega}\subseteq  B_x$, and $\{\phi_j(x)\}_{j\in \omega}\subseteq B_x^\vee$ be sequences such that:
    \begin{enumerate}
        \item For all finite sequences $(a_i)\subseteq k$, the function
        \[x\mapsto\left\|\sum_ia_iv_i(x)\right\|_x\]
        is measurable;
        \item For all finite sequences $(a_i)\subseteq k$, the function
        \[x\mapsto\left\|\sum_ia_i\phi_i(x)\right\|_x\]
        is measurable;
        \item For all $i,j\in\omega$, the function $x\mapsto \phi_i(x)\big(v_jclackton
        (x)\big)$ is measurable.
    \end{enumerate}
    Then the Gram--Schmidt procedure produces an M-basis $(b_i(x),b_i^*(x))$ on $B_x$ such that the pair 
    \[(B_x,\{b_i(x)\}_{i\in \omega})\]
    is an abstract Banach bundle in $\cat{mBan}^M$.
\end{lemma}
\begin{proof}
    The proof is essentially Markushevich's proof of the existence of M-bases using his modified version of the Gram--Schmidt algorithm found in Theorem 4.59 of\cite{FHH+11}. This calculates the M-basis $(b_i(x),b_i^*(x))_{i\in \omega}$ step by step by defining each pair $(b_i(x),b_i^*(x))$ inductively as below, on odd steps defining $b_i(x)$ first, and on even steps defining $b^*_i(x)$ first. 
    \begin{align*}
        h_i(x)&:=\min\big\{h\in \N:  v_{h}(x)\notin\Span\{b_j(x):j<i\}\big\} \\
    b_i(x)&:=v_{h_i}(x)-\sum_{j<i}\Big[b^*_j(x)\big(v_{h_i}(x)\big)\Big]b_j(x)
\end{align*}
\begin{align*}
    k_i(x)&=\min\{k\in \N: \phi_i(x)\big(v_{k_i}(x)\big)\neq 0\}\\
    b^*_i(x)&:=\frac{1}{\phi_{k_i}(x)\big(b_i(x)\big)}\left(\phi_{k_i}(x)-\sum_{j<i}\Big[\phi_{k_i}(x)\big(b_{j}(x)\big)\Big]b^*_j(x)\right)
    \end{align*}
    The key thing to note is that each of these functions of $x$ above is measurable in $x$.

    It is clear that the family $(B_x,\{b_i(x)\})$ satisfies the first 3 axioms of Definition \ref{Def: Objects}, and the fourth arises because each $b_i(x)$ can be rewritten as a finite span of the $v_i(x)$ with finitely many measurable coefficients. Since one can always express such an expression as a limit of step functions, Axiom 4 arises from the hypothesis that $\left\|\sum_ia_iv_i(x)\right\|_x$ is measurable.
    
\end{proof}

\begin{prop}{\label{Prop: pointwise iso implies iso}}
    Let $(B_x,\{b_i(x)\}_{i\in \omega})$ and $(C_x,\{c_i(x)\}_{i\in \omega})$ be two abstract Banach bundles. Suppose further that the M basis of $C_x$ is almost everywhere a Schauder basis. If $(T_x):(B_x)\to (C_x)$ is a morphism in $\cat{mBan}^M_\omega$ such that each $T_{x}:B_x\to C_x$ is an isomorphism, then $(T_x)$ is an isomorphism.
\end{prop}

\begin{proof}
    Suppose that $(T_x):(B_x)\to (C_x)$ is a measurable map which is a topological isomorphism at each $x\in M$. 
    Let $(T^n_x)$ be a sequence of step basic families that converges to $(T_x)$. Set $(S_x^n)$ to be the family defined by 
    \[S_x^n:=\begin{cases}
        (T^n_x)^{-1} & \textrm{if $T_x^n$ is invertible} \\
        0 & \textrm{otherwise.}
    \end{cases}\]
    While each $S_x^n$ is not necessarily itself basic, it can be approximated arbitrarily well by step basic families by Proposition \ref{Prop: Decent=Everything}.
    It is clear that $(S_x^n)$ converges pointwise to the inverse of $(T_x)$, since for any $x\in M$, some tail of $T_x^n$ consists only of invertible operators. 
\end{proof}

\subsection{$\cat{mBan}_{\N}^M$ is an elementary Quasi-abelian category}

It is routine to show that $\cat{mBan}^M_{\omega}$ is a $\KK$-linear additive category: $\Hom$ sets are enriched in $\cat{Vect}$, the zero family is the zero object, and because $\omega$ is infinite, it admits finite biproducts through some (in fact, any) choice of bijection $\coprod_{i=1}^n\omega\cong \omega$. 

In general however, one cannot expect that $\cat{mBan}^M_\omega$ is abelian, even when $M=\{*\}$, for the same reason that the category of Banach spaces is not abelian; while one does have a notion of kernels and cokernels, the natural norms obtained on the image and coimage of a map $f$ need not coincide, or even be equivalent. These categories are instead quasi-abelian.

We restrict our attention to the case $\omega=\N$ and consider the category $\cat{mBan}^M_\N$ of separable abstract Banach bundles, whose objects are pointwise separable Banach spaces. The rest of the section is dedicated to showing that $\cat{mBan}^M_\N$ is an (countably cocomplete) elementary\footnote{Formally, it is incorrect to say that the category is elementary, since part of this definition requires cocompleteness. Due to the obvious bound on the `size' of objects, it cannot include uncountable coproducts. This is minor, and we discuss this in due time.} quasi-abelian category\footnote{For a full definition of the adjective `elementary' and `quasi-abelian', we point the reader to Definitions 2.1.10 and 1.1.3 of \cite{Schneiders1999} respectively.}.

\begin{lemma}\label{Lem: ker and coker}
    The category $\cat{mBan}^M_\N$ contains kernels and cokernels.
\end{lemma}

\begin{proof}

Let $(T_x):(B_x,\{b_i(x)\}_{i\in \N})\to (C_x,\{c_i(x)\}_{i\in \N})$ be a measurable morphism in $\cat{mBan}^M$. In the Archimedean case, consider the set 
\[\int^{\oplus,1}_MB_x dx:=\left\{(v_x)\in \Hom_{\cat{mBan}^M}(\mathds{1},(B_x)):\int_M\|v_x\|_{B_x}dx<\infty\right\}_{/\sim}\]
where two families $(v_x)\sim (w_x)$ are considered to be equivalent if they are almost everywhere equal. It is routine to show that this is in fact a Banach space with norm $\|\cdot\|_{\oplus,1}:=\int_M\|\cdot\|_{B_x}dx$.
Let $Q$ be a countable dense subfield of $\KK$ as before. Because each $B_x$ is separable, and because there is a countable base $\Bc$ on the topology on $M$, it can be easily seen that the set 
\[\Sc(Q,M):=\left\{\sum_{i=1}^\infty \mathds{1}_{M_i} v_i: v_i\in \Span_Q\{b_i(x)\}, M_i\in \Bc \right\}\]
consisting of step functions taking values $v_i\in\Span_{Q}\{b_i\}$ and with preimages in the countable base $\Bc$, is both countable and dense, so that this Banach space is separable. 

In the non-Archimedean case, $\int^{\oplus,1}_MB_x dx$, as defined is not a Banach space, so we instead simply use this notation to consider the closure of the countable set $\Sc(Q,M)$ under the non-Archimedean $\sup$-norm. After that, the rest of the proof works just the same as in the Archimedean case.

We can, without loss of generality, assume that $\|T_x\|\leq 1$ for all $x$. If not, one can replace $T_x$ with the scalar multiple 
\[\tilde{T}_x:=\begin{cases}
    T_x & \textrm{if } \|T_x\|\leq 1\\
    \frac{T_x}{\lambda_x} &\textrm{with $\lambda_x$ a fixed scalar of norm $\|T_x\|$ otherwise}
\end{cases}\footnote{Of course, in the Archimedean case, one can just take $\lambda_x:=\|T_x\|$}\]
where it is easy to see that $\tilde{T}_x$ is the pointwise limit of some sequence $\tilde{S}_x^n$ of step basic families modified in the same way from some sequence of families that approach $T_x$ in the limit, and since the association $T_x\mapsto \tilde{T}_x$ is a linear isomorphism in $\cat{mBan}^M$, and consequently $(\operatorname{co-})\ker T_x=(\operatorname{co-})\ker\tilde{T}_x$ for each $x\in M$, we can make this swap without any loss of generality.

Because $\|T_x\|\leq 1$, we find that the family $T_x$ gives rise to a bounded map between separable Banach spaces
\[T:\int^{\oplus,1}_MB_xdx\to \int^{\oplus,1}_MC_xdx\]
and we see that the kernel and cokernel (arising from the quasi-abelian structure of $\cat{Ban}$), $\ker T\subseteq \int^{\oplus,1}_MB_xdx$ (respectively. $\coker T=\int^{\oplus,1}_MC_xdx/\overline{\im T}$) is therefore separable and equipped with a natural inclusion $\iota: \ker T \hookrightarrow \int^{\oplus,1}_MB_xdx$ (resp. surjection $\pi: \int^{\oplus,1}_MC_xdx\twoheadrightarrow \coker T$).

As with all separable Banach spaces, the Markushevich basis existence theorem guarantees that $\ker T$ (resp. $\coker T$) has an $M$-basis $(\kappa_i,\kappa_i^*)$ (resp. $(\gamma_i,\gamma_i^*)$).

In the case of kernels, we can view each $\kappa_i$ as an element of $\int^{\oplus,1}_MB_xdx$, and thus see there is a Bochner measurable representative $(\kappa_i(x))\in (B_x)$ in the equivalence class of $\kappa_i$, and in the cokernel case, each $\gamma_i$ can be lifted to a Bochner measurable representative $(\gamma_i(x))\in (C_x)$.

In both situations, since the span of $(\kappa_i)_{i\in \omega}$ (resp. $(\gamma_i)_{i\in \omega}$) is dense in $\ker T$ (resp. $\coker T$), we can choose each $\kappa_i(x)$ (resp. $\gamma_i(x)$) so that, outside a measure zero set $Z\subseteq M$, the span of $(\kappa_i(x))_{i\in \N}$  (resp. $(\gamma_i(x))_{i\in \N}$ is dense in $\ker T_x\subseteq B_x$ (resp. $\coker T_x\twoheadleftarrow C_x$).

For those $z\in Z$ where $(\kappa_i(z))_{i\in \omega}$ does not span $\ker T_z$, we can take any $M$ basis $b_i(z)$ we like, and because the measure $\mu$ on $M$ is complete, the combined family is measurable. The cokernel works similiarly. 
We thus find $\{\kappa_i(x)\}\subseteq B_x$ (resp. $\{\gamma_i(x)\}\subseteq C_x$) is a Bochner- measurable family, whose closed span is $\ker T_x$ (resp. surjects onto $\coker T_x$).

The next step is to construct families $\tilde{\kappa}_i(x)$ (resp. $\tilde{\gamma}_i(x)$) so that, for each $x\in M$, they form a minimal set whose closed span is $\ker T_x$ (resp. $\coker T_x$). 
This is simply an application of Gram--Schmidt (See Lemma \ref{Lem: Gram--Schmidt}) on the pairs $(\kappa_i(x),b^*_i(x))$ and $(\gamma_i(x),\gamma^*_i(x))$.

Finally, we want to show that the natural inclusion 
\[\iota=(\iota_x):(\ker T_x,\{\tilde{\kappa}_i(x)\})\to (B_x,\{b_i(x)\})\]
given by $\iota_x:\ker T_x\subseteq B_x$ for each $x\in M$
defines a measurable morphism $\ker(T_x)\to (B_x)$. 

Let $\kappa_i^n(x)$ be a sequence of step families in $(B_x)$ that converges pointwise to $\tilde{\kappa}_i(x)$. By Lemma \ref{Lem: Rapid convergence}, we may even assume that each $\kappa_i^n(x)$ has image inside the algebraic span of $\{b_i(x)\}$. 

By Egorov's Theorem, we can throw away an arbitrarily small subset $S$ of $M$, and assume that $\kappa^n_i(x)$ converges uniformly to $\tilde{\kappa}_i(x)$. By considering subsequences, we may also assume the stronger condition
\[\|\kappa_i^n(x)- \tilde{\kappa}_i(x)\|_x<2^{-ni}.\]
Let $\iota_x^n$ be the map that sends $\tilde{\kappa}_i(x)\in \ker T_x$ to $\kappa_i^n(x)\in B_x$.  Firstly, one finds that each $\iota_x^n$
is a step basic family. Secondly, for all $x\in M\setminus S$, and any $(a_i)\in k^{<\omega}$ bounded above by $A$, 

\begin{align*}
    \left\|\iota^n\left(\sum_{i}a_ib_i(x)\right)-\iota\left(\sum_{i}a_ib_i(x)\right)\right\|_x &=\left\|\sum_ia_i(\iota^n-\iota)\big(b_i(x)\big)\right\|_x \\
    &\leq A\sum_i\|\kappa^n_i(x)-\kappa_i(x)\|_x \\
    &< A\sum_i 2^{-in}=A2^{1-n}\xrightarrow{n\to \infty}0.
\end{align*}
from which one can see that
$\iota_x^n\to \iota_x$ in the $\|\cdot\|_x$-norm topology. Since the set $S$ is arbitrarily small, we can extend this pointwise convergence to almost all $x\in M$. This completes the proof. 
\end{proof}

An important property of this construction is that kernels and cokernels are not only determined locally - but they are determined \emph{pointwise}. That is, the Banach spaces $\ker(T_x)$ and $\coker(T_x)$ are determined by the Banach spaces $B_x,C_x$ and the operator $T_x$. When determining $\ker T_x$, both the properties of other operators $T_{x'}$ for $x'\neq x$ and their behaviour in small neighbourhoods around $x$ are completely irrelevant. 
This property is maintained of course by any construction (such as images and coimages) involving kernels and cokernels, and any property of a morphism that can be expressed in terms of kernels and cokernels (such as strictness) - which can thus be checked pointwise.
This is an important fact that will rear its head later on.

\begin{theorem}{\label{Thm: mBan QA}}
     In the category $\cat{mBan}^M_\N$, kernels are stable under pushout and cokernels are stable under pullbacks. Thus, it forms a quasi-abelian category.
\end{theorem}

\begin{proof}
    Forgetting the coherent M-basis, this is true pointwise, because $\cat{Ban}$ is a quasi-abelian category, and the only thing left to do is check that there is a suitable choice of M-basis on the pushout/pullback, respectively. But since such things can be constructed through a combination of biproducts, cokernels, and kernels, this follows from Lemma \ref{Lem: ker and coker}.
\end{proof}

\begin{remark}
As an aside, it is worth pointing out that it is possible that the restriction to separable Banach spaces is unnecessary, though likely further thought is required. 
There are unfortunately (or perhaps interestingly) non-separable Banach spaces with no Markushevich basis, even if they are subspaces of Banach spaces with one. It seems plausible to the authors that if one restricts attention to families of operators $(T_x)$ which are the limit of step \emph{basic} families, then one could perform some transfinite version of the Gram--Schmidt algorithm of Lemma \ref{Lem: Gram--Schmidt} on the M-basis of $(B_x)$ (resp. $(C_x)$) to obtain a suitable M-basis on the (co)kernel. However, this is a complication beyond the scope of our current research, and it is possible that there are technical issues obstructing this strategy. 
\end{remark}

\begin{description}
    \item[Characterisation of $\im$, $\coim$ and strictness:] Since kernels and cokernels are both determined pointwise in $\cat{mBan}^M_{\N}$, it follows that images and coimages can also be determined pointwise, with M-bases determined in the same manner as Lemma \ref{Lem: ker and coker}. Hence, given a morphism 
    $(T_x):(B_x)\to (C_x)$, we have
    \begin{align}
        \coim(T_x)&=(\coim T_x)=(B_x/\ker T_x)\\
        \im(T_x)&=(\im T_x)=\Big(\overline{\{T_xv: v\in B_x\}}\Big)
    \end{align}
    with the pointwise coimage inheriting the norm from $B_x$ and the pointwise image inheriting from $C_x$. 

    Recall that in a quasi-abelian category, a morphism $f:A\to B$ is called \emph{strict} if the natural morphism induced 
    \[\coim(f)\to im(f)\]
    is an isomorphism. It now follows from the above discussion, and from Proposition \ref{Prop: pointwise iso implies iso} that a measurable morphism $(T_x)$ is strict if each $T_x$ is strict. 
        
    We have suppressed the M-bases involved in the notion because they play a secondary role. 
    
\end{description}

\begin{remark}
    We warn that even if $(V_m), (W_m)$ are continuous abstract Banach bundles and $(\phi_m)$ a continuous family, the kernel $\ker(\phi_m)$ need not be continuous. If $M=\R$, and $V_m=W_m=\R$ for all $m$, then the association $m\mapsto m\id_{\R}$ is continuous, but the kernel gives an abstract Banach bundle that has a discontinuity at $m=0$ (which has a $1$ dimensional kernel).
\end{remark}

When one is studying sheaf theory valued in an abelian category $\Ac$, one usually needs the category $\Ac$ to have more structure than that necessitated by abelian-ness. One asks for properties such as cocompleteness, so that one can take (for example) stalks of a sheaf, and to gain the richness of sheaf cohomology, one would also like to have enough injectives. One suitable environment where one guarantees these properties, is when $\Ac$ is a \emph{Grothendieck category} (also known, in Schneiders' terminology \cite{Schneiders1999} as an \emph{elementary abelian category}).

In the quasi-abelian context, the story is similar, where we need to add extra conditions to make it useful for sheaf theory, and as in the abelian case, a good environment is that of elementary quasi-abelian categories (See Definition 2.1.10 of \cite{Schneiders1999}).
These are quasi-abelian categories which are cocomplete and have a set of tiny\footnote{$A$ is tiny if $\Hom(A,\_)$ preserves filtered colimits.} projective objects that generate the category.
The rest of the section is dedicated to proving that, up to a very small caveat regarding uncountable coproducts, $\cat{mBan}^M_{\N}$ satisfies these conditions. 

The first thing to do is note the following result:
\begin{prop}
    The category $\cat{mBan}^M_\N$ is countably cocomplete; that is, it has all countable colimits. 
\end{prop}

\begin{proof}
    Countable coproducts are simple. One simply needs to reorder the basis $\omega \times \omega \cong \omega$. Since cokernels also exist, it follows that $\cat{mBan}^M_{\N}$ has all countable colimits.
\end{proof}

Here is the afore-mentioned caveat. Strictly speaking, because we demand that $(B_x)$ has a countable M-basis, it is not true that $\cat{mBan}^M$ has uncountable coproducts. The most natural attempt to solve this, by enlarging $\cat{mBan}_{\N}^M$ to a union $\bigcup_\lambda\cat{mBan}_{\lambda}^M$ indexed over limit ordinals $\lambda$, unfortunately doesn't seem to work, because of an earlier remark that non-separable Banach spaces may not have an M-basis. A different approach that is more fruitful, is to simply adjoin formal coproducts to $\cat{mBan}_{\N}^M$, and work in this cocompleted category. 

We end this section by demonstrating the existence of a collection of abstract Banach bundles $\Pc\subseteq \Ob(\cat{mBan}^M_{\N})$ which are a tiny projective generating set of the category. That is, $\Pc$ satisfies:
\begin{enumerate}
    \item (Tiny) For all $P\in \Pc$, and all (countable) filtering systems\footnote{For readers not familiar with this terminology, this is just the categorical analogue of a directed set} $\{E_i\}_{i\in I}$, one has
    \[\Hom(P,\varinjlim_{i\in I}E_i)\cong\varinjlim_{i\in I}\Hom(P,E_i)\]
    \item (Projective) For all $P\in \Pc$, and all strict epimorphisms $E\twoheadrightarrow F$, the morphism
    \[\Hom(P,E)\to\Hom(P,F)\]
    is surjective. 
    \item(Generating) For any non-isomorphic monomorphism $m:E\to F$, there is some $P\in \Pc$ and some map $f:P\to F$ which does not factor through $E$. 
\end{enumerate}

\begin{prop}
    Let $\Pc=\sigma(M)=\{A\subseteq M: A\textrm{ is measurable}\}$ be the set of measurable subsets of $M$.
    For each $A\in \Pc$, let $\mathds{1}_A$ be the family 
    \[(V_x,\{b_i(x)\}_{i\in \N})\]
    where $V_x=\KK$ if $x\in A$ and $V_x=0$ otherwise,
    with the natural M-basis $b_1(x)=1\in V_x$.
    Then $\{\mathds{1}_A\}_{A\in \Pc}$ is a strictly generating set of projective tiny objects.
\end{prop}

\begin{proof}
    \begin{description}
        \item[1. (Tiny)] One can describe the functor $\Hom(\mathds{1}_A,\_)$ as that associating an abstract Banach bundle $(B_x)$ to the abstract Banach bundle $(B'_x)$ where $B'_x=\begin{cases}
            B_x & \textrm{if } x\in A \\
            0 & \textrm{otherwise}
        \end{cases}$
        Such a restriction functor is either the zero functor or the identity functor pointwise, and therefore preserves any colimit, which can always be calculated pointwise, since that is true for coproducts and cokernels.
        
        \item[2. (Projective)] Since $\Hom(\mathds{1}_A,\_)$ is pointwise either the identity or the zero functor, it is clearly strongly exact. Therefore $\mathds{1}_A$ is projective. 
        \item[3. (Generating)] Let $(A_x)\to (B_x)$ be some non-isomorphic monomorphism. The open mapping Theorem guarantees that this monomorphism is not surjective, so there is some $x\in M$, and some $v\in B_x$ such that $v\notin A_x$. Let $A=\{x\}$. 
        Then the morphism $\mathds{1}_A\to (B_x)$ given by $1\mapsto v$ does not factor through $(A_x)$.
    \end{description}
\end{proof}

\subsection{Some related categories}

Henceforth, we fix $\omega=\N$ and we suppress the index in the notation, e.g. we write $\cat{mBan}^M$ in place of $\cat{mBan}^M_\N$.

This section, is dedicated to a discussion of some categories that are closely related to $\cat{mBan}^M$. 

\begin{description}
    \item[Bounded abstract Banach bundles:] A category of great importance in this paper is $\cat{bBan}^M=\cat{bBan}^M_\N$. This is the subcategory of $\cat{mBan}^M$ whose morphisms are essentially uniformly bounded. 
    Note that the kernels $(K_x,\iota_x)$ and cokernels $(C_x,\pi_x)$ constructed in Lemma \ref{Lem: ker and coker} naturally lie in $\cat{bBan}^M$ so that $\cat{bBan}^M$ automatically inherits the quasi-abelian structure of $\cat{mBan}^M$, obtaining:
    \begin{theorem}{\label{thm: bBan QA}}
        The category $\cat{bBan}^M_{\N}$ is quasi-abelian.
    \end{theorem}
    \item[Hilbert spaces and finite dimensional vector spaces:] We are also interested in $\cat{mHilb}^M$ (resp. $\cat{bHilb}^M$) - the full subcategories whose objects $(B_x,\{b_i(x)\})$ have norms $\|\cdot\|_x$ which arise from inner-product structures $\langle\cdot,\cdot\rangle_x$; and $\cat{mVect}^M$ (resp. $\cat{bVect}^M$, - the full subcategory whose objects $(B_x,\{b_i(x)\})_{x\in M}$ are all finite dimensional vector spaces. 

All of these categories are closed under finite direct sums and kernels and cokernels, and strict morphisms remain strict in the subcategory, so they are all quasi-abelian. In fact, $\cat{mVect}^M$ is even abelian, since all morphisms in it are strict, although counter-intuitively $\cat{bVect}^M$ is again only quasi-abelian; the morphism 
$$(x\mathds{1}_\C):(\C,\{1\in \C\})_{x\in (0,1]}\to (\C,\{1\in \C\})_{x\in (0,1]} $$ 
in $\cat{bVect}^{(0,1]}$ is not strict, else it would be an isomorphism. 

In both of these categories, life is somewhat easier than in the Banach space case. The norm $\|\cdot\|_x$ on $H_x$ for an abstract Banach bundle $(H_x)\in\cat{mHilb}^M$ is determined completely by Pythagoras' Theorem, and is therefore constant as $x$ varies. 

Similarly, because all norms on $\KK^n$ are equivalent to the $\ell_p$-norm, it turns out that $\cat{mVect}^M$ is equivalent to its full subcategory consisting of the abstract Banach bundles $(B_x)$ whose norms $\|\cdot\|_{x}$ are equal to the $\ell_p$-norm.  

\item[Representations] Let $G$ be a group. A natural category to study is 
\[\cat{mRep}^M(G):=\cat{Fun}(*/G,\cat{mBan}^M).\]

If $G$ has additional structure, such as being a topological group, then one should take the full subcategory of abstract Banach bundles such that each composition $\ev_x\circ R:*/G\to \cat{Ban}$ is a continuous Banach representation. 

If $G$ is the real/complex points of a reductive algebraic group $G=\Gc(\KK)$, and $K$ is a choice of maximal compact subgroup, then one is often interested in decomposing a representation $(\pi,V)$ of $G$ into its $K$-types. In this situation, it is really rather natural to look at abstract Banach bundles whose underlying M-basis gives the set of $K$-finite vectors. One can then define a category 
\[\cat{mRep}^M(G,K):=\left\{\begin{array}{c} (\pi_x,B_x,\{b_i(x)\}_{i\in \omega})\in \cat{mRep}^M(G) \\ : \Span \{b_i(x)\}=\{v\in B_x: v\text{ is $K$-finite}\} \\ \text{ and }\pi_x:G\to \Aut(B_x) \text{ is continuous}\end{array}\right\}.\]

\end{description}

\section{Functorial Direct Integrals}

In this section, we introduce the measurable sections functor as a functor from $\cat{mBan}^M$ to vector spaces. We then introduce the direct integral as a functorial notion. We also prove that both these functors satisfy precise exactness properties, and reflect exactness.

\subsection{The measurable sections functor}

Let $(A_x),(B_x)$ be two abstract Banach bundles in $\cat{mBan}^M_w$. Inside the Hom-set, one has the subspace of almost everywhere zero families:
\[Z_{\mu}\big((A_x),(B_x)\big)=\Big\{(T_x):\mu\big(\{x\in M: T_x\neq 0\}\big)=0\Big\}\subseteq \Hom\big((A_x),(B_x)\big) \]
which satisfies nice properties such as
\[Z_{\mu}\big((A_x),(B_x)\big)\circ Z_{\mu}\big((B_x),(C_x)\big)\subseteq Z_{\mu}\big((A_x),(C_x)\big)\]
which allows us to form the quotient category $\widetilde{\cat{mBan}}^{(M,\mu)}_w$ whose objects are those of $\cat{mBan}^M_w$ and whose hom sets are 
\[\Hom_{\widetilde{\cat{mBan}}^{(M,\mu)}_w} \big((A_x),(B_x)\big):=\Hom_{\cat{mBan}^{M}_w}\big((A_x),(B_x)\big)/Z_{\mu}\big((A_x),(B_x)\big)\]
This category is also quasi-abelian when $w=\N$. 

For clarity, given two abstract Banach bundles $(A_x),(B_x)\in \cat{mBan}^M_w$, we will use $\Hom\big((A_x),(B_x)\big)$ to mean $\Hom_{\cat{mBan}^{M}_w} \big((A_x),(B_x)\big)$ and 
$\widetilde{\Hom}\big((A_x),(B_x)\big)$ to mean $\Hom_{\widetilde{\cat{mBan}}^{(M,\mu)}_w} \big((A_x),(B_x)\big)$.

\begin{definition}{\label{Def: measurable sections}}
    Let $\big\{(V_x,\{b_i(x)\}_{i\in\omega}\big\}\in \cat{mBan}^M_{w}$. We call a family of vectors $\{v_m\}$ a \emph{(weakly) measurable section} if, for each $i\in \omega$, the function $M\rightarrow k: m\mapsto b_i^*(m)(v_m)$ is measurable.
    
    For an abstract Banach bundle $(V_x,\{b_i(x)\})$, we define the \emph{measurable sections functor} as the functor from $\cat{mBan}^M_w$, factoring through the quotient category $\widetilde{\cat{mBan}}^{(M,\mu)}_w$, to the category of vector spaces $\cat{Vect}$ defined on objects by
    \[\int^mV_mdm=\{\{v_m\}: \textrm{the family is measurable}\}/\thicksim\]
    where we consider two families $\{v_m\}\sim \{v'_m\}$ to be equivalent whenever $v_m=v'_m$ for almost all $m\in M$.
    and which sends measurable morphisms $(T_x)$ to the underlying linear map $(v_x)\mapsto (T_xv_x)$. 
\end{definition}

\begin{remark}
    Please note that weak measurability is a priori different from what we will call Bochner measurable sections; that is, sections that are pointwise limits of families $(v_x^n)_{x\in M}$ that are step functions, in the sense that $M$ admits a countable decomposition $\{M_k\}_{k\in \N}$ such that each $(v^n_x)$ is constant on each $M_k$. 

    Bochner measurability implies weak measurability and, because the Banach spaces $(B_x)$ are all separable, it turns out that the inverse also holds true. 
    This is most easily seen in the case where each $B_x=B$ a constant abstract Banach bundle and $\KK$ is Archimedean, when it is a clear consequence of Pettis' measurability theorem. We postpone the more general proof to Appendix \ref{App: A}, so as not to interrupt the flow.
\end{remark}

It remains to prove (see Lemma \ref{Lem: morphisms and Meas. families}) that measurable sections are indeed sent to measurable sections.

\begin{lemma}{\label{Lem: morphisms and Meas. families}}
    Let $(T_x):(B_x)\to (C_x)$ be a morphism in $\cat{mBan}^M_\omega$. If $(v_x)\in(B_x)$ is a measurable family, then $(T_x(v_x))\in(C_x)$ is a measurable family.
\end{lemma}

\begin{proof}
    It suffices to show that for every $j\in\omega$ the function $c_j^*(x)\big(T_x(v_x)\big)$ is measurable as a function of $x$. As $T_x$ is the limit of families of step functions, it follows that $c_j^*(x)\big(T_x(v_x)\big)$ is measurable if $c_j^*(x)\big(S_x(v_x)\big)$ is measurable for any step family $S_x$. Furthermore, by restricting to a smaller subset $M'\subseteq M$ if necessary, it is enough to prove this for constant families $(S_x)$. 

    In this situation, the functional $c^*_j(x)S_x$ is constant, in the sense that for each $i\in \omega$, the quantity $c^*_j(x)S_xb_i(x)$ is independent of $x$. 

    When $c^*_j(x)S_x$ is a finite linear combination of the $b_i^*(x)$, it is clear that this is measurable. Otherwise, it is in the weak${}^*$-closure of a set of functionals that are, so it follows that $c^*_j(x)S_xb_i(x)$ is the limit of measurable functions, and is thus measurable. 
\end{proof}

\begin{remark}
        In the case $\omega=\N$, there is a much more succinct proof as follows. 
        Let $\mathds{1}$ be the abstract Banach bundle $(k,\{b_i(x)\})$ in $\cat{mBan}^M_\N$ where $b_i(x)=1\in \KK$ if $i=0$  and $b_i(x)=0$ otherwise. 
        Let $Z_\mu((B_x),(C_x))\subset\Hom((B_x),(C_x))$ be the subspace of homomorphisms that are almost everywhere trivial. 
        By Appendix \ref{App: A}, measurable families of vectors $(v_m)\in (V_m)$ are in a one-to-one correspondence with the elements of $\Hom(\mathds{1},(V_m))/Z(\mathds{1},(V_m))$, so that the Lemma follows from Lemma \ref{lem: composition}.
\end{remark}

One can think of the measurable sections functor $\int^m:\cat{mBan}^M_w\to \widetilde{\cat{mBan}}^{(M,\mu)}_w\to \cat{Vect}$ as a process with two steps: first quotient the morphisms of the category by the almost trivial morphisms, then compose with $\widetilde{\Hom}(\mathds{1},\_)$. 

\begin{example}{\label{Ex: measurable sections sim product}}
    Let $(B_x)_{x\in M}$ be a abstract Banach bundle, indexed by a measure space $M$ with the counting measure, and suppose $|\omega|\geq\big(\sup_x\dim(B_x)\big)\times{|M|}$. Suppose that the functions $b_i(x)$ are such that $b_i(x)=0$ for all but a single $x$ (which is possible, because of our assumption on the cardinality). Then 
    \[\int^mB_x=\prod_{x\in M} B_x\]
\end{example}
\begin{example}
    If $B_x=\C$ for all $x$, and $b_0(x)=1$ for all $x\in M$, then 
    \[\int^mB_x=\{\textrm{measurable functions $f:M\to \C$}\}/\thicksim\]
\end{example}

We recall the differing notions of exactness in quasi-abelian categories (see \cite{Schneiders1999} \S1.1). 

\begin{definition}
    A null sequence 
    \[A\xrightarrow{f} B\xrightarrow{g} C\]
    is strictly (co)exact if the natural map $\im(f)\to\ker(g)$ is an isomorphism and if $f$ (resp. $g$) is strict.
\end{definition}

\begin{definition}
    \begin{enumerate}
        \item A functor $F:\cat{C}\to \cat{D}$ is exact if it sends strictly short exact sequences 
        \[0\to A\to B\to C\to 0\] 
        to strictly short exact sequences. 
        \item A functor $F:\cat{C}\to \cat{D}$ is strictly (co)exact if it sends strictly (co)exact sequences 
        \[A\to B\to C\] 
        to strictly short (co)exact sequences.
        \item A functor $F:\cat{C}\to \cat{D}$ is strongly exact if it is both strictly exact and strictly coexact.
    \end{enumerate}
\end{definition}

\begin{example}
    The forgetful functor $\cat{Ban}\to \cat{Vect}$ is strictly exact, but not strictly coexact because one requires strictness of $f:A\to B$ to guarantee that $A/\ker(f)\cong \im(f)\cong \ker g $.  
    On the other hand, the forgetful functor $\cat{Ban}\to \cat{Norm}$ is strongly exact, because the notions of image and coimage are not changed by this functor (more precisely, they are always Banach spaces, whether regarded as Banach or normed spaces). 
\end{example}

\begin{prop}{\label{Prop: Measurable exactness}}
The association $\cat{mBan}^M_\N\to \cat{Vect}:\{V_m\}\mapsto \int^mV_m $ is a strictly exact functor, which we call the \emph{measurable sections functor}.
That is, $\int^m$ transforms exact sequences
\[(A_x)\xrightarrow{f}(B_x)\to (C_x)\] 
with $f$ strict to exact sequences. 
The functors $\ev_m:\cat{mBan}^M_\N\rightarrow \cat{Ban}$ are all strongly exact.

\end{prop}

\begin{proof}
    The strong exactness of $\ev_m$ follows directly from the pointwise construction of the kernel and cokernel of Lemma \ref{Lem: ker and coker}, and the fact that strictness is a pointwise condition. 

    To show that $\int^m$ is strictly exact, consider the exact sequence 
    \[(A_x)\xrightarrow{f}(B_x)\xrightarrow{g}(C_x)\]
    with $f$ strict. 
    Then, pointwise, the morphisms $A_x/\ker(f_x)\to \ker(g_x)$ are always isomorphisms, so in particular they isomorphisms of $\cat{mBan}^M$ by Proposition \ref{Prop: pointwise iso implies iso}, making the sequence
    \[\int^mA_x\xrightarrow{\int^mf}\int^mB_x\xrightarrow{\int^mg}\int^mC_x\]
    strictly exact (strictness arising because all morphisms in $\cat{Vect}$ are strict).   
\end{proof}

\begin{prop}{\label{Prop: Reflecting Exactness}}
    The functors $\ev_m$ and $\int^{m}$ `reflect' exactness in the following precise way. Let 
    \begin{equation}{\label{Eq: Sequence}}
        (A_x)\xrightarrow{f_x} (B_x)\xrightarrow{g_x} (C_x)
    \end{equation}
    be a sequence with $g_xf_x=0$. Then:
    \begin{itemize}
        \item If all sequences \footnote{i.e. - applying the functors $\{\ev_x\}_{x\in M}$}
        \[A_x\xrightarrow{f_x} B_x\xrightarrow{g_x} C_x\]
        are strictly (co)exact then so is Sequence (\ref{Eq: Sequence}).
        \item If the sequence 
     \[\int^mA_x\xrightarrow{\int f_x} \int^mB_x\xrightarrow{\int g_x} \int^mC_x\]
     is exact, then there is a measure zero set $Z\subseteq M$ for which Sequence (\ref{Eq: Sequence}) is strictly exact in $\cat{mBan}^{M\setminus Z}_\N$.
    \end{itemize}
\end{prop}

\begin{proof}
    The first part follows directly from the pointwise description of strict morphisms and kernels/cokernels. 

    By Proposition 2.1.8 (also 1.3.23) of \cite{Schneiders1999}, 
    because $\mathds{1}$ is a tiny projective object that generates $\widetilde{\cat{mBan}}^{(M,\mu)}_\N$ in the sense of Defintion 2.1.5 of \cite{Schneiders1999}, the functor $\widetilde{\Hom}(\mathds{1},\_)$ reflects exact sequences so that 
    $$(A_x)\to (B_x)\to (C_x)$$
    is strictly exact in $\widetilde{\cat{mBan}}^{(M,\mu)}_\N$. 

    Therefore, the natural map $\iota_x:\im f_x\to \ker g_x$ is an isomorphism in $\widetilde{\cat{mBan}}^{(M,\mu)}_\N$ with inverse $p_x$. The set $Z=\{x\in M:p_x\circ \iota_x\neq \id\}\cup \{x\in M:\iota_x\circ p_x\neq \id\}$ is therefore measure zero, and so we obtain that $A_x\to B_x\to C_x$ is exact for all $x\in M\setminus Z$.
    To check that $f_x$ is almost everywhere strict is similar, obtaining the result after the application of Proposition \ref{Prop: Measurable exactness}.    
\end{proof}

\subsection{Integrability}\label{the functorial direct integral}

We now turn our attention to direct integrals. As with the Hilbert direct integral, we expect a few different properties:
\begin{itemize}
    \item  $\int_M^\oplus B_x dx$ should be a complete Banach space. 
    \item The integral needs to be functorial in $(B_x)$.
    \item For any open neighbourhood $U$ of $x$, 
    $\int_U^\oplus B_x$ should approximate $B_x$, and the approximation should improve as $\mu(U)\to 0$. 
    \item Our notion should generalise the Von-Neumann's construction in the Hilbert space case.
\end{itemize}

For the property of functoriality, the category $\cat{mBan}^M_w$ is not quite good enough, as the following example demonstrates.

    \begin{remark}
        Let $M=(0,1)$, and let $V_m=\C$ for each $m\in M$. The endomorphism $(\phi_m)$ of $(V_m)\in\cat{mBan}^M_w$ given by $\phi_m(v)=m^{-691}v$ is a measurable morphism (indeed, it is even continuous), but it sends the family $(1)_{m\in M}$ with $L^2$-norm $||1||_{p}=1$ to a family $(m^{-691})_{m\in M}$ with infinite $L^2$-norm.
    \end{remark}
    To rectify this example, we need to restrict the types of morphisms to those that are uniformly bounded. 

\begin{definition}
      We say a morphism $(T_m)_{m\in M}$ is \emph{uniformly bounded} if 
    
    \[\esssup_{m\in M}\{||T_m||\}<\infty\]
    where the norm $\|T_m\|$ is the operator norm.

    We set the category $\cat{bBan}^M_w\subseteq\cat{mBan}^M_w$ (resp. $\cat{bHilb}^M_w\subseteq\cat{mHilb}^M_w$) as the category whose objects are simply all abstract Banach bundles, and whose morphisms are those morphisms $(T_m)$ whose norms are uniformly bounded. When $w=\N$, just like $\cat{mBan}^M_\N$, these are all an elementary quasi-abelian categories. Furthermore, the $\Hom$-spaces all exhibit a natural seminorm, given by the essential supremum above, and are complete with respect to this seminorm. 
    Of course, like with the space $L^\infty$, the seminorm of a morphism $(T_x)$ is zero if and only if $T_x=0$ for almost all $x\in M$. 
\end{definition}

\begin{lemma}{\label{lem: uniformly strict}}
    In the category $\cat{bBan}^M_\N$, a morphism $(T_x):(A_x)\to (B_x)$ is strict if and only if there are constants $\alpha,\beta\in (0,\infty)$ such that the induced norms on the pointwise coimage and images $\|\cdot\|_{\coim T_x}$ and $\|\cdot\|_{\im T_x}$ are \emph{uniformly} bounded
    \[0<\alpha\leq \frac{\|\cdot\|_{\im T_x}}{\|\cdot\|_{\coim T_x}}\leq \beta\]    
\end{lemma}

\begin{proof}
    We can decomplete so that the coimage and image live on the same dense vector space, so that the difference is only on the norm, inherited from that on $A_x$ and $B_x$ respectively. That $(\coim T_x)\isomto (\im T_x)$ in $\cat{bBan}^M_\N$ means that the identity map on the underlying vector space is uniformly bounded with uniformly bounded inverse with respect to these different norms, and we can take these bounds to be $\beta$ and $\alpha^{-1}$ respectively. 
\end{proof}

We are now ready to define $L^p$-direct integrals. Recall that $M$ is a topological space equipped with a Radon measure $\mu$ on its Borel $\sigma$-algebra.

\begin{definition}
    Let $p\in [1,\infty]$ and $\Big\{(B_x,\{b_i(x)\}_{i\in w})\Big\}_{x\in M}\in\cat{bBan}^M_w$. If $\KK=\R$ or $\C$ is Archimedean and $p<\infty$, the $L^p$-direct integral is the normed vector space
    \[\int^{\oplus,p}_MB_xdx:=\left\{f(x)\in \int^mB_xdx:  \int_M\|f(x)\|_x^p dx<\infty\right\}.\]
    Whether $\KK$ is Archimedean or non-Archimedean and $p=\infty$,  one can similarly define
    \[\int^{\oplus,\infty}_MB_xdx:=\left\{f(x)\in \int^mB_xdx:  \esssup_{m\in M}\|f(x)\|_{B_x} <\infty\right\}\]
    and its sibling 
    \[\int^{\oplus,0}_MB_xdx:=\left\{(v_x)\in \int^{\oplus,\infty}_MB_xdx: \forall \epsilon>0, \mu\{x\in M: \|v_x\|_{B_x}>\epsilon\}<\infty \right\}\]
    whose sections are those that vanish at $\infty$.
    
    All the above are complete normed spaces (even an inner product space, in the case $p=2$ and all $B_x$ are Hilbert spaces) with norm $\|\cdot\|_p$ given by
    \[\|(v_x)\|_p=\left(\int_M\|v_x\|_x^pdx\right)^{1/p}\quad \text{for }p<\infty;\]
    \[\|(v_x)\|_\infty=\|(v_x)\|_0=\esssup_x\left\{\|v_x\|_x\right\}.\]
    When unclear, we may also denote the norm $\|\cdot\|_p$ by $\|\cdot\|_{\oplus,p}$, or $\|\cdot\|_{\oplus}$.   
\end{definition}
\begin{description}[font=\normalfont\itshape\space]
    \item[Remark 1:]  In the case $B_x=B$ is a constant abstract Banach bundle, then 
    \[\int^{\oplus,p}_MB_xdx=L^p(M;B)\]
    is the Bochner $L^p$ space with $B$-coefficients. 

    When each $B_x$ is a Hilbert space and $p=2$, we can call this the Hilbert-direct integral and $\int^{\oplus}_MB_xdx$ coincides with the usual direct integral as defined in texts such as \cite{Nielsen80} or Chapter 8 of \cite{DeitmarPHA}. 
    \item[Remark 2:] In the non-Archimedean case, because non-Archimedean Banach spaces are required to satisfy an ultrametric inequality, only the direct integrals $\int^{\infty_M}$ and $\int^{\oplus,0}_M$ make sense, and are therefore defined. 
    \item[Remark 3:] There is a big difference between the $p<\infty$ and $p=\infty,0$ cases. It is a fact that if $p<\infty$, and $M$ is a Polish space, then $\int^{\oplus,p}_MB_xdx$ is separable. 

    However, in the $p=\infty$ case, $\int^{\oplus,\infty}_MB_xdx$ is rarely separable, because whenever there is an uncountable collection of measurable sets $\{U_i\subseteq M\}_{i\in\aleph_1}$, any pair of the uncountably many functions $\mathds{1}_{U_i}b_1(x)$ are distance $1$ from each other, and therefore there can be no countable dense subset. 
    A similar story holds for $\int^{\oplus,0}$, except one requires the subsets $U_i$ to have finite measure for the counterexample to work.

    In fact, this counterexample shows that $\int^{\oplus,\infty}B_xdx$ is separable only if $M$ has countably many measurable sets, and $\int^{\oplus,0}B_xdx$ is separable only if $M$ has countably many subsets of finite measure.
    These further imply that $M$ is equivalent to a finite(resp. countable) measure space with counting measure.    
\end{description}
 
    \begin{lemma}
        If $(T_x):(V_x)\rightarrow (W_x)$ is a bounded morphism, then it induces a bounded morphism of normed vector spaces 
        \[\int^{\oplus,p} V_xdx\xrightarrow{\int T}\int^{\oplus,p} W_xdx.\]
        If $T_x$ is an isometry for almost all $x\in M$, then $\int T$ is an isometry. 
    \end{lemma}
    \begin{proof}
        Since the question of measurability is sorted in Lemma \ref{Lem: morphisms and Meas. families}, the only thing to check is that whenever $(v_x)$ is an $
        L^p$ family, then so is $(T_xv_x)$, which we only prove in the $p<\infty$, Archimedean case, all other cases are just as simple.         
        If $||T_x||\leq M$ for almost all $x$, then
        \[\int_M||T_xv_x||_p^pdx\leq \int_M||T_x||^p \ ||v_x||_p^{p} dx \leq M^p\int ||v_x||^{p}
        _pdm<\infty\]
        so $(T_xv_x)$ is an $L^p$ family. Similarly, whenever $T_x$ is an isometry, $||T_xv_x||_p=||v_x||_p$, so the lemma easily follows. 
    \end{proof}

    \begin{remark}
        It follows that $\int^{\oplus,p}_M$ is a functor from the subcategory $\cat{bBan}^M_w$ to $\cat{Ban}$.
    \end{remark}

\begin{prop}{\label{Prop: DI strong exact}}
     The functor $\int^{\oplus,p}_M:\cat{bBan}^M_\N\to \cat{Ban}$ is strongly exact.
\end{prop}

\begin{proof}
    Consider a null sequence 
    \[(A_x)\xrightarrow{(T_x)}(B_x)\xrightarrow{(S_x)}(C_x)\]
    in $\cat{bBan}^M_\N$. 
    The topologies on $\im (\int T_x) $ and $\ker(\int S_x)$ are given by the respective norms 
    \[\left[\int_M\|\cdot \|_{\im T_x}^pdx\right]^\frac{1}{p} \qquad \text{and} \qquad \left[\int_M\|\cdot \|_{\ker S_x}^pdx\right]^\frac{1}{p}\]
    when $p$ is finite and 
    \[\sup_{x\in M}\|\cdot \|_{\im T_x} \qquad \text{and} \qquad \sup_{x\in M}\|\cdot \|_{\ker S_x}\]
    when $p=0,\infty$.
    If the sequence $A_x\xrightarrow{T_x}B_x\xrightarrow{S_x}C_x$ is exact for each $x\in M$, then the image $\im T_x$ and kernel $\ker S_x$ correspond to the same subspace of $B_x$, and their norms $\|\cdot\|_{\im T_x}$ and $\|\cdot\|_{\ker S_x}$ are identical, both being expressible as the restriction of the norm on $B_x$. Thus, the norms on $\im(\int T_m)$ and $\ker(\int S_x)$ are identical, and the sequence is exact. 

    To show that $(T_x)$ strict implies $\int T_x$ strict, we note by Lemma \ref{lem: uniformly strict} that there are $\alpha,\beta\in (0,\infty)$ such that the two norms $\|\cdot\|_{\coim T_x}$ and $\|\cdot\|_{\im T_x}$ satisfy 
    \[0<\alpha\leq \frac{\|\cdot\|_{\im T_x}}{\|\cdot\|_{\coim T_x}}\leq \beta.\]
    independent of $x\in M$. 
    
    Thus, we obtain the inequalities:
    \[0<\int_M\alpha^p\|f(m)\|^p_{\coim T_m}dm\leq \int_M\|f(m)\|^p_{\im T_m}\leq \int_M\beta^p\|f(m)\|^p_{\coim T_m} \]
    from which it is clear that 
    \[0<\alpha\|f(m)\|_{\coim \int T_m}\leq\|f(m)\|_{\im \int T_m}\leq \beta\|f(m)\|_{\coim \int T_m}. \]
\end{proof}

\subsection{Coefficients-Continuity and the continuous sections functor}

In functional analysis, the non-separability of $L^{\infty}(M)$ is often a barrier to usefulness. It therefore makes sense to study the related closed subset of continuous functions. This is especially true in the non-Archimedean situation, where essentially the only norm available is the $L^\infty$ norm. This reintroduces separability. 
We therefore want to briefly talk about the notion of continuous sections to produce a better tool for the non-Archimedean situation. 

\begin{definition}{\label{Def: continuous family}}
    Let $\{B_x,\{b_i(x)\}_{i\in \omega}\}\in \cat{bBan}_{\omega}^M$. We say that an abstract Banach bundle $\{B_x\}$ is a \emph{continuous abstract Banach bundle} if every function of the form

\[x\mapsto \left\|\sum_{i}a_ib_i(x) \right\|_{B_x}\]
is continuous whenever $(a_i)\in k^{<\omega}$. 
\end{definition}

\begin{remark}\label{Rem: joint continuity}
    Notice that an abstract Banach bundles $\{B_x\}$ is continuous if and only if the function 
    \begin{align*}
        N:k^{<\omega}\times M&\to \R^{\geq 0}\\
        ((a_i),x)&\mapsto \left\|\sum_i a_ib_i(x)\right\|_x
    \end{align*}
    is continuous, when $\KK^{<\omega}$ is endowed with the $\ell^1$-norm $\|(a_i)-(a_i')\|_{k^{<\omega}}:=\sum_i |a_i-a_i'|$. In fact, the family of functions 
    \begin{align*}
        N_x:k^{<\omega}&\to k \\
         (a_i)&\mapsto N((a_i),x)
    \end{align*}
    is equicontinuous as $x$ varies in $M$. To see this, notice that if $(a_i),(a_i')\in k^{<\omega}$ and $x\in M$, then
    $$\Big|N_x((a_i))-N_x((a_i'))\Big|\le\Big|\sum_i(a_i-a_i')b_x(i)\Big|\le \Big|\Big| (a_i)-(a_i')\Big|\Big|_{k^{<\omega}}$$
    and recall that if $X,Y,Z$ are metric spaces and $f:X\times Y\to Z$ is such that 
    \begin{itemize}
        \item for every $x\in X$ the function $f(x,-)$ is continuous;
        \item the family of functions $f(-,y)$ is equicontinuous as $y$ varies in $Y$,
    \end{itemize}
    then $f$ is jointly continuous. 
\end{remark}

\begin{definition}
    Let $(B_x)$, $(C_x)$ be two bounded abstract Banach bundles in $\cat{bBan}^M$. We say a uniformly bounded measurable morphism $(T_x):(B_x)\to (C_x)$ is \emph{coefficients-continuous basic} if for all $i,j\in \omega$, the matrix coefficient
    \[x\mapsto c_j^*(x)T_xb_i(x)\]
    is continuous, and if for all $i$, there is some $N_i\gg0$ such that for all $j\geq N_i$, the matrix coefficient $c_j^*(x)T_xb_i(x)\equiv 0$.

    We say a morphism $T_x$ is \emph{coefficients-continuous} if there is a sequence $T_x^n$ of coefficients-continuous basic morphisms such that
    \[\|T_x-T_x^n\|_x\to 0 \]
    uniformly in $x\in M$. 

    A measurable section $(v_x)\in (B_x)$ is of course coefficients-continuous if the corresponding morphism $(k)\to(B_x)$ is coefficients-continuous. 
\end{definition}

We offer the warning to the reader that even if a measurable section $v_x$ is coefficients-continuous, its norm function $x\mapsto \|v_x\|_x$ may well not be continuous, as demonstrated by the following counterexample. 
\begin{example}
    Let $M=[0,1]$, and let $B_x=\begin{cases}
        \ell^1 &\text{if } x\leq 1/2 \\
        \ell^\infty &\text{if } x>1/2
    \end{cases}$
    with the usual bases. 
    Then $v_x=b_1(x)+b_2(x)$ is coefficients-continuous, but its norm function is $\|v_x\|_x=\begin{cases}
        2 &\text{if } x\leq 1/2 \\
        1 &\text{if } x>1/2
    \end{cases}$.
\end{example}

\begin{prop}
    If $(B_x)$ is a continuous abstract Banach bundle, then whenever $(v_x)$ is a coefficients-continuous section, the norm function 
    \begin{align}
        N':M&\to \R^{\geq0 } \\
        x&\mapsto \|v_x\|_x
    \end{align}
    is continuous.
\end{prop}

\begin{proof}
    There is a sequence $(v_x^n)\in (B_x^n)$ of coefficients-continuous sections that converges uniformly to $(v_x)$. Therefore, by the remark following Definition \ref{Def: continuous family}, the composition 
    \begin{align}
        M\to k^n\times M\xrightarrow{N}\R^{\geq 0} \\
        x\mapsto \big((v^n_x),x\big)\to \|v^n_x\|_x
    \end{align}
    is continuous. Then $\|v_x\|_x$ is the uniform limit of continous functions, proving the proposition.
\end{proof}

\begin{prop}
    Let $(B_x,\{b_i(x)\}_{i\in \omega}),(C_x,\{c_i(x)\}_{i\in \omega})$ and $(D_x,\{d_i(x)\}_{i\in \omega})$ be abstract Banach bundles of $\cat{bBan}^M$.
    If $(T_x):(B_x)\to(C_x)$ and $(S_x):(C_x)\to (D_x)$ are coefficients-continuous morphisms, then so is the composition $(S_x\circ T_x)$. 
\end{prop}

\begin{proof}
    Suppose $T_x$ and $S_x$ are both coefficients-continuous. If $(S_x)$ and $(T_x)$ are coefficients-continuous basic morphisms, then the identities
    \[d_k^*(x)S_xT_xb_i(x)=\sum_j d_k^*(x)S_x c_j(x)c_j^*(x)T_xb_i(x)\]
    follow because all but finitely many terms in this sum are zero. 
    It thus follows that $S_xT_x$ is coefficients-continuous basic. 

    If $S_x$ and $T_x$ are not basic, 
    then there are sequences of coefficients-continuous basic morphisms $S_x^n$ and $T^n_x$ that converges uniformly to $S_x$ and $T_x$ respectively. Therefore, for $n$ large enough,
    \begin{align*}
        \|S_xT_x-S^n_xT_x^n\|&\leq \|S_xT_x-S_xT_x^n\|+\|S_xT_x^n-S_x^nT_x^n\| \\
        &\leq \|S_x\|\|T_x-T_x^n\|+\|S_x-S_x^n\|\|T_x^n\|\\
        &\leq \|S_x\|\|T_x-T_x^n\|+\|S_x-S_x^n\|\big(\|T_x\|+1\big)
    \end{align*}
    because $\|T_x^n\|\leq \|T_x\|+1$ for all $n\gg0$. Because $(T_x)$ and $(S_x)$ are uniformly bounded, it follows that this sequence converges to zero uniformly in $x$. Therefore, $(S_xT_x)$ is coefficients-continuous. 
\end{proof}

\begin{remark}
    Suppose $(T_x):(B_x)\to (C_x)$ is a coefficients-continuous basic morphism. For each $n\in \N$, one can restrict $T_x$ to $(B_x^n)$ Then $T_x|_{B_x^n}$ can be viewed as a $N\times n$ matrix (with $N=\max\{N_1,\cdots, N_n\}$).
    The map 
    \[M\to M_{N\times n}(k): \qquad x\to T_x|_{B_x^n}\]
    is then continuous, because all coefficients of the finite dimensional vector space are continuous. 
\end{remark}

\begin{definition}
    The category $\cat{ccBan}^M$ of measurable abstract Banach bundles with coefficients-continuous morphisms is the subcategory of $\cat{bBan}^M$ whose objects are the continuous abstract Banach bundles and whose morphisms are the closed subspace of $\Hom_{\cat{bBan}^M}\big((B_x),(C_x)\big)$ consisting of the coefficients-continuous morphisms. 
    The continuous sections functor is the functor
    \[\int^{\oplus,c}_M:=\Hom_{\cat{ccBan}^M}\big((k),\bullet\big):\cat{cBan}^M\to \cat{Ban}.\]
\end{definition}

That $\int^{\oplus,c}_M$ is a functor follows from the previous proposition. 
\section{Algebraic and Analytic Properties of Direct Integrals}

In what follows, we will always assume that $\omega=\N$, and supress the index notation in the categories $\cat{bBan}^M:=\cat{bBan}^M_\N$, $\cat{mBan}^M:=\cat{mBan}^M_\N$ etc.

Also note that in sections 4.1, 4.3 and 4.4, due to technical issues of separability, we often need to restrict to the situation $p\in [1,\infty)$,(or even $p\in (1,\infty)$) and thus inherently require that the base field $\KK$ is Archimedean. 

\subsection{Disintegrations} 

Naturally, one wants to ask the question of how one can vary between index spaces $M$. This notion was tackled by Michael Wendt in \cite{hilbertsheaves}, though in a slightly different context with model-theoretic motivations. We outline the ideas here.

\begin{definition}
    We define the category $\cat{Disint}$ of disintegrations as the category whose objects are Polish spaces with Borel measures $(M,\Fc,\mu)$ and whose morphisms between two measure spaces $(X,\Fc_X,\mu)\to (Y,\Fc_Y,\nu)$ are pairs
    \[(f, \{\nu_y\})\]
    with $f:X\rightarrow Y$ a measurable function and $\{\mu_y\}$ a family of measures on the sub-measure spaces $f^{-1}(y)\subseteq X$ such that for any measurable subset $U\subseteq X$
    \[\mu(U)=\int_{f(U)}\mu_y(U\cap f^{-1}(y))d\nu(y).\]
    Such a morphism is known as a \emph{disintegration}.
\end{definition}
\begin{remark}
    Given a measurable function $f:(M,\mu)\to (M',\mu')$ of Polish spaces, a consequence of the disintegration theorem (see \cite{WikiDisint} for a statement and \S 10.4, \S 10.6 of \cite{Bogachev07} for a proof) is the necessary and sufficient condition for the existence of a disintegration compatible with $f,\mu$ and $\mu'$ is that $f$ reflects measure zero sets. In other words, for $U\subseteq M'$ measurable and $\nu(U)=0$, then $\mu(f^{-1}(U))=0$.     
\end{remark}

\begin{remark}
    Our category $\cat{Disint}$ is a full subcategory of the category considered in Section 2 of \cite{hilbertsheaves}. This is not a significant restriction, since for our purposes, it makes sense to consider Radon measures, and it is a technical requirement that we need the topologies in question to have a countable base.
\end{remark}

\begin{lemma}
    Let $(f,\{\mu_y\}):(M,\mu)\to (N,\nu)$ be a disintegration. For each $p\in [1,\infty)$, we obtain a functor
\begin{align*}
    \int^{\oplus,p}_f:\cat{bBan}^M&\to \cat{bBan}^{N} \\
    (V_x,\{b_i(x)\})_{x\in M}&\mapsto \left(\int^{\oplus,p}_{f^{-1}(y)}V_{x}d\mu_{y}(x),\{b'_{ij}(y)\}\right)_{y\in N}
\end{align*}
where the the M-basis $\{\tilde{b}_{ij}(y)\}$ chosen can be constructed functorially from $\{b_i(x)\}$.  
\end{lemma}

\begin{proof}
    Pointwise, we already know how to construct $\int^{\oplus,p}_{f^{-1(y)}}B_xd\mu_y(x)$, the issue only being finding a sequence of M-bases for each $y\in N$.
    To construct the M-basis, choose a countable base $\{M_j\}_{j\in J}$ of the topology of $M$, indexed by $J$, and let $\{b_{ij}(y)\}_{(i,j)\in \omega\times J}$ be a collection of functions in $\int^{\oplus,p}_{f^{-1}(y)}B_x $ defined by
    \begin{align*}
        b_{ij}(y): f^{-1}(y)&\to \bigcup_{x\in f^{-1}(y)}B_x \\ 
        x&\mapsto \begin{cases}
        b_i(x) & \text{if } x\in M_j \\
        0& \text{otherwise}
    \end{cases}
    \end{align*}
    The $\{b_{ij}(y)\}_{i,j}$ give a countable collection inside $\int^{\oplus,p}_{f^{-1}(y)}B_xd\mu_y(x)$ whose span is dense, and automatically comes along with a set 
    \[\left\{b_{ij}^*(y): \int^{\oplus,p}_{f^{-1}(y)}V_xd\mu_y(x)\ni \phi\mapsto \frac{1}{\mu_y(M_j\cap f^{-1}(y))}\int_{M_j\cap f^{-1}(y)}b^*_i(x)\phi(x)d\mu_y(x)\right\}_{(i,j)\in \omega\times J}\]  
    which separates the points of $\int^\oplus_{f^{-1}(y)}V_x$. (One interprets $b_{ij}^*(y)=0$ in the case $\mu_y(M_j\cap f^{-1}(y))=0$).
    By choosing a bijection $\omega\times J\cong \omega$, and performing the Gram--Schmidt algorithm, one can transform $\{\tilde{b}_{ij}(y),\tilde{b}^*_{ij}(y)\}$ into a Markushevich basis, and complete the construction.

    We have now completed the construction of $\{b'_{ij}\}$, and it remains only to show that (up to isomorphism) the object is choice-free - namely, that the integral $\int^{\oplus,p}_f$ is independent of the choice of countable base of $M$ and the choice of bijection $\omega\times J\cong \omega$. 
    
    Let $\{M_j\}_{j\in J},\{M_j'\}_{j\in J'}$ be two choices of countable base of the topology on $M$, and let $s:\omega\times J\xrightarrow{\sim}\omega$ (resp. $s':\omega\times J'\xrightarrow{\sim}\omega$ be corresponding choices of bijections on the index sets. 

    We show that the map
    \[(\id_y)_{y\in N}:\left(\int^{\oplus,p}_{f^{-1}(y)}V_x, \{\tilde{b}_{ij}(y)\}\right)\to \left(\int^{\oplus,p}_{f^{-1}(y)}V_x, \{\tilde{b}'_{ij}(y)\}\right) \]
    is a measurable isomorphism. 
    Firstly, we argue that change in basis determined by the choices $s$ and $s'$, and the Gram--Schmidt algorithm are measurable, since each element $\tilde{b}_{ij}(y)$ is replaced by a finite linear combination of such vectors determined algorithmically via a measurable process (see Lemma \ref{Lem: Gram--Schmidt}). 
    Thus, it is sufficient to show that the maps 
    $y\mapsto (\tilde{b}')^*_{st}(y)\tilde{b}_{ij}(y)$ 
    are all measurable. 
    We have
    \begin{align*}
        (\tilde{b}')^*_{st}(y)\tilde{b}_{ij}(y)&= \frac{1}{\mu_y(M'_t\cap f^{-1}(y))}\int_{M'_t\cap f^{-1}(y)}(b')^*_s(x)\mathds{1}_{M_j}b_i(x)d\mu_y(x)\\
        &=\begin{cases}
            0 & \text{if } s\neq i \\
            \frac{\mu_y(M'_t\cap f^{-1}(y)\cap M_j)}{\mu_y(M'_t\cap f^{-1}(y))} &\text{otherwise}
        \end{cases}
    \end{align*}
    which is clearly a measurable function of $y$. 
\end{proof}

More generally, we have:
\begin{theorem}
    For each fixed $p\in [1,\infty)$, direct integration is a pseudo-functor from the category of disintegrations to the category of (small) categories
    \[\int^{\oplus,p}:\cat{Disint}\to \cat{CAT}.\]
\end{theorem}
In this context, this means with have the following data:

\begin{enumerate}
    \item A map of objects $M\mapsto  \cat{bBan}^M$;
    \item An association sending any disintegration $f:M\to N$ to the functor \[\int^{\oplus, p}_f:\cat{bBan}^M\to \cat{bBan}^N;\]
    \item The identity map $\id:M\to M$ is associated to the identity functor $\id_{\cat{bBan}^M}$;
    \item for any $f:M\to N$, $g:N\to P$, an isomorphism, 
    \[\eta_{g,f}:\int^{\oplus,p}_g\circ \int^{\oplus,p}_f\Rightarrow\int^{\oplus,p}_{g\circ f}\]
    which is natural in $f$ and $g$;
    \item The functors $\int^{\oplus,p}_{f},\int^{\oplus,p}_{f}\circ\int^{\oplus,p}_{\id_M}, \int^{\oplus,p}_{f\circ\id_M} ,\int^{\oplus,p}_{\id_N}\circ\int^{\oplus,p}_{f}$and $\int^{\oplus,p}_{\id_N\circ f}$ are all exactly the same functor;
    \item The natural isomorphism $\eta_{g,f}$ is associative in that the diagram
\[\begin{tikzcd}[cramped]
	{\int^{\oplus,p}_h\circ(\int^{\oplus,p}_g\circ\int^{\oplus,p}_f)} && {(\int^{\oplus,p}_h\circ\int^{\oplus,p}_g)\circ\int^{\oplus,p}_f} \\
	{\int^{\oplus,p}_h\circ\int^{\oplus,p}_{g\circ f}} && {\int^{\oplus,p}_{h\circ g}\circ\int^{\oplus,p}_{f}} \\
	{\int^{\oplus,p}_{h\circ (g\circ f)}} && {\int^{\oplus,p}_{(h\circ g)\circ f}}
	\arrow["{\id}",equals, from=1-1, to=1-3]
	\arrow["{\id_{\int h}\circ\eta_{g,f}}", Rightarrow, from=1-1, to=2-1]
	\arrow["{\eta_{h\circ g}\circ \id_{\int f}}"', Rightarrow, from=1-3, to=2-3]
	\arrow["{\eta_{h,g\circ f}}", Rightarrow, from=2-1, to=3-1]
	\arrow["{\eta_{h\circ g,f}}"', Rightarrow, from=2-3, to=3-3]
	\arrow[equals, from=3-1, to=3-3]
\end{tikzcd}\]
    commutes.
\end{enumerate}

Most of these properties are obvious (and much stronger than the general definition of a pseudofunctor - a relic of the fact that $\cat{Disint}$ is a $1$-category). 
In fact, it is only properties (4) and (6) that must be affirmed. 

\begin{proof}
    Let $f:M\to N$ and $g:N\to K$ be two disintegrations of Polish spaces. It is clear that the underlying Banach spaces of 
    \[\int^{\oplus,p}_g\int^{\oplus,p}_fV_x ~~ \text{and}~~ \int^{\oplus,p}_{g\circ f}V_x\]
    are identical, and it only remains to show that the family of identity maps in these spaces is measurable - a straightforward exercise. This is clearly a natural choice, and the diagram in part(6) is clearly commutative, since all maps are pointwise identity operators. 
\end{proof}

This functor is compatible with the direct integral in the sense that one can identify 
\[\int^{\oplus}_M=\int^{\oplus}_{M\to *}\]

\subsection{Continuous Abstract Banach Bundles and Local Isometric Approximation}

The defining algebraic property of the direct sum is its universal property;  morphisms $X\to \bigoplus_i X_i$ are in one-to-one correspondence with families of morphisms $\{X\to X_i\}_i$. The direct integral is different, in that there may be, in general, \emph{no} morphisms $X\to \int^{\oplus}X_i$ at all; a clear example given by the non-existence of representation homomorphisms between any one-dimensional representation of $\R$ and the regular representation $L^2(\R)\cong\int^{\oplus,2}_{i\R}\pi_{\alpha}$.
The fundamental reason behind this observation is that direct integrals do not remember the value of a function at any particular point, only its equivalence class of functions modulo measure zero-equivalence.

However, not all hope is lost. The most important property of integration is, of course, the fundamental theorem of calculus. One way of formulating this theorem is the following. 
For any \emph{continuous} function $f:\R\to \R$, 
\[\lim_{\epsilon_1,\epsilon_2\to 0^+}\frac{1}{\epsilon_1+\epsilon_2}\int_{x-\epsilon_2}^{x+\epsilon_1}f(t)dt=f(x).\]
This gives us hope that direct integration might, in some sense, approximate $B_x$ as we consider a `limit' of Banach spaces $\int^{\oplus,p}_UB_tdt$ over a sequence of shrinking neighbourhoods $U_n\subseteq M$ forming a neighbourhood basis of $x$ (so $\bigcap_n U_n=\{x\}$). The only question remaining is: in what sense does this happen?

There is an elementary obstruction, illustrated with the following example. 
\begin{example} Let $M=[0,1]$, and $B_x=k$ with M-basis $b_1(x)=1$ for all $x\in M$.

It is clear that $\int^{\oplus,p}_{[1/2-\epsilon, 1/2+\epsilon]}k=L^p([1/2-\epsilon, 1/2+\epsilon])$ is always infinite dimensional for all $\epsilon>0$, while $B_{1/2}=k$ is one dimensional, so any `limit' as $\epsilon\to 0$ is strictly larger than $B_{1/2}$. 
\end{example}

This example demonstrates that there is an obstruction to the existence of any isomorphism between $B_x$ and a `limit' $\varinjlim_{U\ni x}\int_UB_tdt$, but also hints at the important fact we prove in this section that $B_{1/2}$ can be approximately isometrically embedded in $\int_UB_tdt$, and this approximation improves to an arbitrary degree as the domain $U$ shrinks. 

Recall from Section 3.3, that an abstract Banach bundle $\{B_x,\{b_i(x)\}_{i\in \omega}\}\in \cat{mBan}_{\omega}^M$ is said to be \emph{continuous} if, for all $(a_i)\in k^{<\omega}$, every function of the form
\[x\mapsto \left\|\sum_{i}a_ib_i(x) \right\|_{B_x}\]
is continuous, and recall Remark \ref{Rem: joint continuity}, which tells us the function 
\begin{align*}
        N:k^{<\omega}\times M&\to k\\
        ((a_i),x)&\mapsto \left\|\sum_i a_ib_i(x)\right\|_x
    \end{align*}
    is continuous, when $\KK^{<\omega}$ is endowed with the $\ell^1$-norm $\|(a_i)-(a_i')\|_{k^{<\omega}}:=\sum_i |a_i-a_i'|$.

We also make the following observation. Let $x\in M$, and let $U$ be an (open) neighbourhood of $x$. 

There is a linear map 
\begin{align*}
    I_U:B_x^\circ&\to \int_U^{\oplus,p}B_t \\
    b_i(x)& \mapsto \frac{1}{\mu(U)^{1/p}}(b_i(t))_{t\in U}.
\end{align*}
In general, this map may not be bounded (for example, in the case $\ell^p(\N)\to \int_{p-\epsilon}^{p+\epsilon}\ell^t(\N)dt$ these maps are unbounded). However, one can restrict this map locally to $B_x^n=\Span\{b_1(x),\cdots ,b_n(x)\}$, on which $I_{U,n}=I_U|_{B^n_x}$ is  a continuous injection, and therefore, there are bounds 
\[m_{U,n}:=\inf_{\|v\|_x\leq 1} \|I_{U,n}(v)\|_{\oplus}, \qquad M_{U,n}:=\sup_{\|v\|_x\leq 1} \|I_{U,n}(v)\|_{\oplus}\]
where $\|\cdot \|_{\oplus}$ indicates the norm on $\int_U^{\oplus,p}B_t$,  so that
\[0< m_{U,n}\|v\|_x\leq \|I_{U,n}(v)\|_{\oplus}\leq M_{U,n}\|v\|_x\]
for all non-zero $v\in B^n_x$.

\begin{theorem}{\label{thm: approximate isometries}} Let $p\in [1,\infty]$ and $\KK$ is Archimedean or non-Archimedean.
    Suppose that $B_t$ is a continuous abstract Banach bundle, $x\in M$ and $\{U\}$ is a neighbourhood basis of $x$. Then for any $n\in \N$, the lower and upper bounds 
    $M_{U,n},m_{U,n}$ defined above tend to $1$ as $U$ shrinks to the point $\{x\}$. In other words, $I_{U,n}$ approximates an isometry, and this approximation improves arbitrarily as $U$ shrinks. 
\end{theorem}

\begin{proof}
    Let $C=\{(a_i)\in k^n:|a_i|\leq1\}$. By Remark \ref{Rem: joint continuity}, the function 
    $N:C\times M\to k$ given by 
    \[N:((a_i),t)\mapsto \|\sum_{i=1}^na_ib_i(t)\|_t\]
    is continuous.

Consider the equality
    \[\left\|I_{U,n}\left(\sum_{i=1}^n a_ib_i(x)\right)\right\|^p_{\oplus}=\frac{1}{\mu(U)}\int_U\left\|\sum_{i=1}^n a_ib_i(t)\right\|_t^pdt\]
    As $U$ shrinks to $x$, the quantity on the right converges to 
    \[\left\|\sum_{i=1}^n a_ib_i(x)\right\|_x^p\]
    for any fixed $(a_i)$. Since the set $C=\{(a_i)\in k^n:|a_i|\leq1\}$ is compact, one can find, for any $\epsilon_1$, a finite subset $\{(\tilde{a}_i)\}$ of $C$ such that any expression 
    \[\frac{1}{\mu(U)}\int_U\left\|\sum_{i=1}^n a_ib_i(t)\right\|_t^pdt\] 
    is $\epsilon_1$-close to the same expression with one of the finitely many $(\tilde{a}_i)$.  
    
    One can thus conclude that the convergence is uniform in $(a_i)\in C$, which in turn implies that $m_{U,n},M_{U,n}\to 1$ as $U$ shrinks to $x$. 
\end{proof}

\begin{corollary}\label{Cor: approximate isometries}
    For any measurable abstract Banach bundle, there is a measure zero subset $Z\subseteq M$ such that for all $z\in M\setminus Z$, $M_{U,n}$ and $m_{U,n}$ as above converge to 1, so that $I_{U,n}$ approximates an isometry.
\end{corollary}

\begin{proof}
    This is essentially a consequence of Lusin's theorem, which states that given any measurable function $f$ on a Borel space $X$ with finite measure, and given any $\epsilon$, there is a subset $Z$ of measure at most $\epsilon$ such that $f|_{X\setminus Z}$ is continuous. 

    The point is that we can take a countable dense subset of $C^n$, (namely, $Q^n\cap C^n$), put some order on the corresponding norm functions of these $v_n(x)\in Q^n\cap C^n$, and apply Lusin's theorem to obtain a subset $Z=\bigcup_n Z_n$ of measure at most $\epsilon=\sum_{n=1}^{\infty} \epsilon/2^n$ away from which all norm functions $\|v_n(x)\|_x$ are continuous. Because all rational norm functions are continuous, one can extend to all norm functions to obtain that $(B_x)_{M\setminus Z}$ is a continuous abstract Banach bundle, and so by the previous theorem, we obtain that for all $z\in M\setminus Z$, the relevant convergence holds.  
\end{proof}

\subsection{Duality}

Let $(B,\{b_i\})$ be a Banach space with a Markushevich basis $(b_i,b_i^*)$.
We let $B^\vee$ be the dual of $B$; that is, the Banach space of bounded functionals. Consider the subspace $B^*\subseteq B^\vee$ defined to be the subset of $B^\vee$ that is the \emph{Norm} closure of $\Span\{b_i^*\}$ (rather than the weaker weak-$^*$ closure). 

Because the set $\{b_i\}_{i\in \N}$ is the set of biorthonormal functions of $\{b_i^*\}_{i\in \N}$, and their algebraic span is norm dense inside $B\subseteq (B^*)^\vee\subseteq B^{\vee\vee}$, Goldstine's theorem guarantees that $\{b_i\}$ separates the points of $B^*$, so that $(b_i^*,b_i)_{i\in \N}$ is a Markushevich basis of $B^*$. 

\begin{definition}
    We call the pair $(B^*,\{b^*_i\}_{i\in \omega})$ the \emph{shrinking dual} of $(B,\{b_i\}_{i\in \omega})$. 
\end{definition}

Note that, by definition, an M-basis $\{b_i\}$ of $B$ is shrinking if and only if $B^*=B^\vee$. Chief examples include the standard bases on $\ell^p$ for $ p\in (1,\infty)$, (not $\ell^1$!) and the space of sequences $c_0$ that converge to $0$.
On the other hand, $\ell^1$ has no shrinking basis. With the standard basis, we have $(\ell^1)^\vee=\ell^\infty$ and $(\ell^1)^*=c_0$. 

Because $B^{**}$ and $B$ are both identified as the norm closure of $\Span\{b_i\}$, it is easy to see that the shrinking double dual is idempotent, that is $(B^{**},\{b_i\})=(B,\{b_i\})$.

\begin{lemma}{\label{lem: Norm convergence}}
Suppose $M$ has finite measure.
    Let $(B_x)\in\cat{bBan}^M$ and let $(v_x)\in \int^{\oplus,p}_MB_x$. Then there is a sequence of measurable sets $(U_k)$, a double sequence of integers $(s^n_k)_{n,k\in \N}$,  and a sequence $(v_n(x))\in \int^{\oplus,p}_MB_x$ such that: 
    \begin{itemize}
    \item $U_{k+1}\subseteq U_{k}$ for all $k\in \N$,
    \item The measure $\mu(U_k)\leq 1/10^k$
        \item the restriction of $v_n(x)$ to $U_{k-1}\setminus U_k$ lies inside
        \[\int^{\oplus,p}_{U^n_k}B^{s^n_k}\]
        \item $(v_n(x))$ converges to $(v_x)$ in norm.
    \end{itemize}
    In particular, for each $x\in M$ the vector $v_n(x)\in \Span\{b_i(x)\}$ is in the algebraic span of the $\{b_i(x)\}$. 
\end{lemma}

\begin{proof}
    By Lemma \ref{Lem: Rapid convergence}, we know there is a sequence $f^1_n(x)$ which converges to $v_x$ pointwise. By Egorov's theorem, there is a subset $U_1\subseteq M$ on which the convergence of $\|f^1_n(x)-v_x\|_x$ to $0$ is uniform and whose measure is $\mu(M\setminus U_1)\leq 1/10$. 
    Let $(s^n_1)$ be an increasing sequence of integers such that
\[\|f^1_{s^n_1}(x)-v_x\|_x\leq 1/2^n\]
for all $x\in M\setminus U_1$ and all $n\in \N$. 

By induction, construct $f^k_n(x)\in \Span\{b_i(x)\}$ defined on $U_{k-1}$, which converges uniformly outside some set $U_k$ of measure $\mu(U_k)\leq 1/10^k$.  Let $(s^n_k)$ be an increasing sequence of integers such that
\[\|f^k_{s^n_k}(x)-v_x\|_x\leq 1/2^n\]
for all $x\in U_{k-1}\setminus U_k$. 

It is clear that $\bigcap_k U_k$ has measure zero, and that the family $v_n(x)$ defined by 
\[v_n(x):=\begin{cases}
    f^k_{s^n_k}(x) &\text{if } x\in U_{k-1}\setminus U_k \\
    0 &\text{if } x\in \bigcap_k U_k
\end{cases} \] converges almost everywhere and uniformly to $v_x$. Thus, one finds
\[\|v_x-v_n(x)\|_{\oplus} \leq \mu(M)/2^n\]
for each $n$ and thus that $v_n(x)$ converges to zero. 
\end{proof}

\begin{prop}{\label{Prop: Direct integral duals}}  Assume the base field $\KK$ is Archimedean. 
Let $p\in [1,\infty]$ and $q=\frac{p}{p-1}$ be the H\"older dual. 

Let $(B_x,\{b_i(x)\}_{i\in \N})_{x\in M}$ be an abstract Banach bundle in $\cat{bBan}^M$ and let $(B_x^*,\{b^*_i(x)\}_{i\in \N})_{x\in M}$ be its shrinking dual.
\begin{enumerate}
      \item There is a pairing
    \begin{align*}
        \langle\cdot,\cdot\rangle:\int^{\oplus,p}_MB_x\times \int^{\oplus,q}_MB_x^*&\to k \\
        (v_m),(\phi_m)&\mapsto \int_M\phi_m(v_m)dm
    \end{align*}
    which induces an isometric linear map
    \[D:\int^{\oplus,q}_MB_x^*\to \left(\int^{\oplus,p}_MB_x\right)^\vee.\]

    \item Assume, in addition, that:
    \begin{enumerate}
        \item $M$ is $\sigma$-finite;
        \item $p\neq \infty$ (so that the base field $\KK$ is Archimedean);
        \item almost all $(B_x,\{b_i(x)\})$ are shrinking.
    \end{enumerate}
    Then $D$ is an isomorphism. 
\end{enumerate}  
\end{prop}

\begin{proof}[Proof of part 1.]
    We first want to show that $D$ is an isometry, which follows from classical ideas by showing that $\|(\phi_x)\|\leq \|D(\phi_x)\|\leq \|(\phi_x)\|$ for all families $(\phi_x)$. The second inequality follows from the classical H\"older-inequality:
    \begin{align*}
        \|D(\phi_x)\|_{\int^{p,\vee}}& =\sup_{\|(v_x)\|_p\leq 1}\left|\int_M \phi_xv_xdx\right| \\
        &\leq \sup_{\|(v_x)\|_p\leq 1}\int_M |\phi_xv_x|dx \\
        &\leq \sup_{\|(v_x)\|_p\leq 1}\int_M \|\phi_x\|_{B_x^\vee}\|v_x\|_{B_x}dx \\
        &\leq \sup_{\|(v_x)\|_p\leq 1}\left(\int_M \|\phi_x\|^qdx\right)^{\frac{1}{q}}\left(\int_M \|v_x\|^pdx\right)^{\frac{1}{p}} \\
        &=\|(\phi_x)\|_{\int^q}
    \end{align*}
    As for the first inequality, we merely have to find a measurable family $(v_m)$ (dependent on $\epsilon$) with $\|v_m\|=1$ such that $|\phi_mv_m|>\|\phi_m\|-\epsilon$. Indeed, this can be done with a step family by taking some close approximation of $\phi_m$.
\end{proof}

\begin{proof}[Proof of part 2.]
It remains only to show that $D$ is surjective under the additional hypotheses.

Let $\lambda:\int^{\oplus,p}_MB_x\to k$ be a linear functional. 
Assume for now that $M$ has finite measure. For each $v\in k^{<\omega}$, define an complex valued association from measurable sets on $M$ by
\[\mu_v:U\mapsto \lambda(v\mathds{1}_U)\]
where $v\mathds{1}_U$ denotes the step family equal to $v$ for $x\in U$ and $0$ otherwise. 

This association is in fact a complex valued measure. It is countably additive because if $U:=\coprod_iU_i\subseteq M$, then
\begin{align*}\sum_{i=1}^{\infty}\mu_v(U_i)&=\lim_{N\to \infty}\sum_{i=1}^{N}\lambda(v\mathds{1}_{U_i})\\
&=\lim_{N\to \infty}\lambda(v\mathds{1}_{\coprod_{i=1}^NU_i})\\
&=\mu_v(U)
\end{align*}
The above limit evaluates to $\lambda(v\mathds{1}_U)$ because $\lambda$ is continuous and $\sum_{u=1}^{N}\mathds{1}_{U_i}$ converges to $\mathds{1}_U$ in the $L^p$ norm (here, $p\neq \infty$ is important) and because $M$ has finite measure. 

By the Radon-Nikodym theorem, there is complex function $\lambda_x(v):M\to \C$ such that
\begin{equation}{\label{Eq: Radon-Nikodym}}
    \lambda(v\mathds{1}_U)=\int_U\lambda_x(v)d\mu(x).
\end{equation}

As an aside, we can now extend the existence of $\lambda_x(v)$ to the case $M$ is a $\sigma$-finite measure space, by dividing $M$ into countably many disjoint finite measure subsets, and applying the Radon--Nikodym theorem to each subset. 

The family
\[\lambda_x(v)=\varinjlim_{U\ni x} \frac{\lambda(v\mathds{1}_U)}{\mu(U)}\]
can be expressed as a limit over neighbourhoods of $x$ (at least, for almost all $x\in M$), it follows that the association $v\mapsto \lambda_x(v)$ is linear in $v$. 
Furthermore, we can demonstrate that $\lambda_x$ is continuous with respect to the topology on $B_x^\vee$ through the following calculation
\begin{align*}
        |\lambda_x(v)|&=\varinjlim_{U\ni x}\frac{1}{\mu(U)}|\lambda(v\mathds{1}_U)| \\
        &\leq \|\lambda\| \limsup_{U\ni x} \frac{\|v\mathds{1}_U\|}{\mu(U)} =\|\lambda\|\|v\|_x
    \end{align*}
    where the last equality arises because the inclusion $B_x\to \int_U^{\oplus,p}B_x$ approximates an isometry by Corollary \ref{Cor: approximate isometries}. Thus, $\lambda_x$ can be uniquely extended to an element of $B_x^\vee$.

    We first claim that the family $(\lambda_x)_{x\in M}$ is Bochner measurable, and lies in $\int^{\oplus,q}_MB_x^*$. 

    Since each function $x\mapsto \lambda_x(v)$ is measurable, it is clear that $\lambda_x$ is a weakly measurable family, so an application of Theorem \ref{Thm: Separable measurablity} shows that $(\lambda_x)$ is a Bochner measurable family of $(B_x^*)$. Also, since $\|\lambda_x\|_{B^*_x}\leq \|\lambda\|$, it is clear that 
    \[\int_M\|\lambda_x\|_{B_x^*}^qd\mu(x)\leq \|\lambda\|^q\]
so that $\lambda_x\in \int^{\oplus,q}_MB_x^*$.

Thus, we have a family of functionals $\lambda_x$, and by Equation (\ref{Eq: Radon-Nikodym}), we can demonstrate that, for any step family $(v_x)=\sum_{i=1}^{\infty}v_i\mathds{1}_{U_i}\in \int^{\oplus,p}_MB_x$,
\begin{align*}
  D(\lambda_x)(v_x)&=\int_M\lambda_x(v_x)d\mu \\
   &=\sum_i\int_{U_i}\lambda_x(v_x)d\mu \\
   &=\sum_i\lambda(v_i\mathds{1}_{U_i})\\
   &=\lambda (v_x).
\end{align*}
Thus, since $\lambda$ and $D(\lambda_x)$ agree on a dense subset of $\int^{\oplus,p}_MB_x$, we conclude that $\lambda=D(\lambda_x)$. 

\end{proof}

\subsection{Integral Kernel representations of bounded operators}
In this section, we assume the base field $\KK$ is Archimedean. 
Here, we want to prove a result that amounts to the existence of integral kernels. In our application, we want to be able to view morphisms $T:\int^{\oplus,p}_MX_m\to \int^{\oplus,q}_NY_n$ between direct integral spaces as an integral against a family of morphisms $k_T(m,n):X_m\to Y_n$, such as in Equation \ref{Eq: integral kernel}. However, one must be careful. 

    If one allowed a more general notion of direct integral than we have defined in this paper, then there may well be no local representation of such a bounded operator $T$. 
    
    In a simple example, let $C^1(-1,1)$ be the space of continuously differentiable functions, which can be expressed as some kind of direct integral $\int^\oplus_{(-1,1)}\C$, by taking all measurable sections on which the norm $f\mapsto \|f\|_\infty+\|f'\|_{\infty}$ is well defined. 
    The Schwartz kernel theorem guarantees that any operator $T:C^1(-1,1)\to \C$ can be expressed as a distributional integral kernel, but a distributional kernel need not be anywhere close to being a function.
    As a counterexample, consider the bounded operator $T:f\mapsto f'(0)$, which can be expressed as an integral
    \[Tf=\int_{-1}^1\delta'(m)f(m)dm\]
    against a distribution. The differential $\delta'(m) dm$ does not take the form $k(m)d\mu(m)$ for any pair with $k(m)$ a function, because if it did, the support of the measure $\mu$ would necessarily be the set $\{0\}$, taking the form of a (scalar multiple of) the counting measure, in which case the integral (\ref{Eq: integral kernel}) takes the form 
    \[\int_{-1}^1k_T(m)f(m)d\mu(m)=k_T(0)f(0)\]
    which visibly does not depend on $f'(0)$. 

\begin{theorem}{\label{Thm: integral kernels}}
    Assume the base field $\KK$ is Archimedean. 
    Suppose $M,N$ are locally compact Polish spaces with Borel measures. 
    Suppose $(X_m)\in \cat{bBan}^M$ and $(Y_n)\in \cat{bBan}^N$
    are abstract Banach bundles such that almost all $Y_n$ are shrinking. Let $p\in [1,\infty)$, $q\in(1,\infty)$, and let $T:\int^{\oplus,p}_MX_m\to \int^{\oplus,q}_NY_n$ be a bounded operator between direct integral spaces. 
    
    Then there exist:
    \begin{enumerate}
        \item a measure $\hat{\mu}_T$ on $M\times N$;
        \item a disintegration $(\hat{\mu}_{T,n}:\Mc(M)\to [0,\infty])_{n\in N}$ of $\mu_f$;
        \item an essentially bounded family $(k_T(m,n)\in \Hom_{\cat{Ban}}(X_m,Y_n))$;
    \end{enumerate}
    such that for any family $(x_m)\in (X_m)$, the operator $T$ is expressible as:
    \begin{equation}{\label{Eq: integral kernel}}
        T:(x_m)_{m\in M}\mapsto \left({\small \int}_M k_T(m,n)x_m d\hat{\mu}_{T,n}(m)\right)_{n\in N}.
    \end{equation}
    
\end{theorem}

\begin{proof}
    Let $T:\int^{\oplus,p}_M X_m\to \int^{\oplus,q}_N Y_n$ be a bounded operator. Identify $\KK^{<\omega}$ with the set $S_k\subseteq \int^{\oplus,p}_M X_m$ of constant families in the finite span. Let $\cdot|_U:\int^{\oplus,p}_M X_m\to\int^{\oplus,p}_UX_m$ denote the natural projection.
    
    Let $\mu_T$ be the $B(k^{<\omega},\int_N^\oplus Y_n)$-valued vector measure defined on the Borel $\sigma$-algebra of $M\times N$ given by
    \[\mu_T(U\times V)=\{x\mapsto T(x|_U)|_V\}\]
    and consider the $[0,\infty]$-valued total variation measure on $M\times N$ given by 
    \[\hat{\mu}_T(U\times V)=\sup_{\substack{\coprod_iU_i\times V_i\subseteq U\times V \\ x_i\in C(U_i)} } \sum_i \|T(x_i)|_{V_i}  \| \]
    where we assume $\bigcup_i U_i$ is a disjoint finite union of measurable subsets. 
    This measure, by the disintegration theorem, exhibits a disintegration $(\hat{\mu}_n)_{n\in N}$ (defined almost everywhere) along the fibres $M_n:=M\times \{n\}\subseteq M\times N $ such that 

    \[\hat{\mu}_T(S)=\int_N \hat{\mu}_n(S\cap M_n)d\nu_N(n). \]

Now, $\mu_T$ is absolutely continuous with respect to $\hat{\mu}_T$, so we want to find the Radon-Nikodym derivative

\[\frac{d\mu_T}{d\hat{\mu}_T}(m,n):=\lim_{\substack{\hat{\mu}_T(Z_i)>0 \\\bigcap_i{Z_i}=\{(m,n)\}}}\mu_T(Z_i)/\hat{\mu}_T(Z_i)\]
where we take the limit in the weak${}^*$ topology.

\begin{claim}
    The Radon-Nikodym derivative $\frac{d\mu_T}{d\hat{\mu}_T}(m,n)$ exists almost everywhere, and we have 
    \[\frac{d\mu_T}{d\hat{\mu}_T}(m,n)\in \Hom(X_m,Y_n)  \textrm{ and } \left\|\frac{d\mu_T}{d\hat{\mu}_T}(m,n)\right\|\leq 1\]
    for $\hat{\mu}_T$- almost all $(m,n)$. 
\end{claim}
\begin{adjustwidth}{1cm}{1cm}
    \begin{proof}
        Since $\mu_T$ is a $B(k^{<\omega},\int^\oplus_N Y_n)$-valued measure, each $\alpha\in k^{<\omega}$ gives rise to a $\int^{\oplus,q}_N Y_n$-valued measure $\mu_{T,\alpha}=\mu_T(\alpha)$, which defines a linear functional on $(\int^{\oplus,q}_VY_n)^\vee=\int^{\oplus,q^*}_V Y_n^*$ for any $V\subseteq N$ by setting
        $$\chi=\sum v_i  1_{A_i}\mapsto \int_{M\times N}\chi d\mu_T:=\sum v_i \mu_T(A_i) $$
        on step functions and extending by continuity. 
        By Proposition \ref{Prop: Direct integral duals}, there is a function $h_{\alpha}\in (\int^{\oplus,q^*}_VY_n^*)^\vee=\int^{\oplus,q}_V Y_n$ such that 
        \[\mu_{T,\alpha}(A)=\int_{M\times N} 1_{A} \  h_\alpha d\widehat{\mu_{T,\alpha}}\]
    where $\widehat{\mu_{T,\alpha}}$ is the total variation of $\mu_{T,\alpha}$. 
    What's more, if we take measurable subsets $V\subseteq V'$, we see that $h_{V,\alpha}(m,n)=h_{V',\alpha}|_{M\times V}(m,n)$, so that $h_{V,\alpha}$ is independent of $V$, and we can take $h_\alpha(m,n)\in Y_n$, independent of $V$.

    Now, $\widehat{\mu_{T,\alpha}}$ is a $[0,\infty]$-valued measure that is absolutely continuous with respect to $\hat{\mu}_T$, because 
    \begin{align}
        \hat{\mu}_T(A)=0&\implies \mu_T(S)=0\ \forall S\subseteq A \\
        &\implies \mu_T(S)(x)=0\ \forall S\subseteq A \\
        &\implies \widehat{\mu_{T,\alpha}}(A)=0.
    \end{align}
    Thus, we can find the usual Radon-Nikodym derivative $\frac{d\widehat{\mu_{T,\alpha}}}{d\hat{\mu}_T}$, 
    and set $h_{1,\alpha}=h_\alpha\frac{d\widehat{\mu_{T,\alpha}}}{d\hat{\mu}_T}\in \int^{\oplus,q}_NY_n$ so that
    \[d\mu_{T,\alpha}=h_{1,\alpha}d\hat{\mu}_T.\]
    These expressions are linear in $x$, so we have the equality of $B(k^{<\omega},\int^{\oplus}_NY_n)$-valued measures
    \[d\mu_T=h_1d\hat{\mu}_T.\]
    allowing us to identify $h_1(m,n)$ with $\frac{d\mu_T}{d\hat{\mu}_T}(m,n)$
    for each $(m,n)$, $h_1(m,n)$ is a linear operator from $\KK^{<\omega}$ to $Y_n$, and because $X_m\to \int^{\oplus,p}_UX_{m'}dm'$ is a near isometry, we can deduce that it is bounded, and thus can view it as a linear operator from $X_m$ to $Y_n$, so that each $h_1(m,n)\in \Hom(X_m,Y_n)$.

    \end{proof}
\end{adjustwidth}

We now claim
\begin{equation}{\label{eq: almosteverywhere}}
    (Tv)_n\overset{a.e}{=}\int_M \frac{d\mu_T}{d\hat{\mu}_T}(m,n) v(m)d\hat{\mu}_{T,n}(m),
\end{equation}
so that $k_T(m,n):=\frac{d\mu_T}{d\hat{\mu}_T}(m,n)$ satisfies the claim in the beginning. 

Let $\phi\in (\int^{\oplus,q}_N Y_n)^{\vee}=\int^{\oplus,q^*}_NY_n^*$, which we can write as an integral $\int_N\lambda(n)d\nu(n)$. We then calculate 

\begin{align}
    \phi\left(\int_M \frac{d\mu_T}{d\hat{\mu}_T}(m,n) v(m)d\hat{\mu}_{T,n}(m)\right) &=\int_N\lambda(n)\int_M \frac{d\mu_T}{d\hat{\mu}_T}(m,n) v(m)d\hat{\mu}_{T,n}(m)d\nu(n) \\
    &=\int_{M\times N}\lambda(n) \frac{d\mu_T}{d\hat{\mu}_T}(m,n) v(m)d\hat{\mu}_{T}(m,n) \\
    &=\int_{M\times N}\lambda(n)v(m)d\mu_{T}(m,n) \\
    &=\phi(T(v))
\end{align}
where we can interpret the final integral expression as a Dobrakov-integral (See. also Definition 2.1 of \cite{Rod06}).
This gives us Equation (\ref{eq: almosteverywhere}).   
\end{proof}

We're mostly interested in building up a theoretical framework to be able to integrate different kinds of algebraic objects together, namely sheaves and representations, which thankfully can be unified in the more general context:

Suppose $\cat{C}=\cat{Fun}(I,\cat{Vect}_k)$ is the category of functors from a small category $I$ to $\cat{Vect}_k$. Let $\cat{bC}^M:=\cat{Fun}(I,\cat{bBan}^M)$ be the category of measurable families of $\cat{C}$-objects, with the direct integral functor $\int^{\oplus}:\cat{bC}^M\to \cat{C}$.

\begin{corollary}{\label{Corol: integral kernels}}
    Assume the base field $\KK$ is Archimedean.
    Let $(X_m)\in \cat{b}\Cc^M, (Y_n)\in\cat{b}\Cc^N$ and let $T:\int^{\oplus,p}_MX_m \to\int^{\oplus,q}_NY_n$ be a bounded operator. Then, for each $i\in I$ there is a family $(k_{T(i)}(m,n)\in \Hom_{\Cc}(X_m,Y_n))$ and a family of measures $(\mu_n)_{n\in N}$ such that 

   \[T(i):(x_m)_{m\in M}\mapsto \left({\small \int}_M k_{T(i)}(m,n)x_m d\mu_n(m)\right)_{n\in N}\]
    
\end{corollary}
    \begin{proof}
        For each $i\in I$, Proposition \ref{Thm: integral kernels} guarantees the existence of the family $(k_{T(i)}(m,n))$ and a measure $\mu_{T(i)}$ on $M\times N$. Furthermore, for each $\phi:i\to j$ in $I$, the corresponding commutative diagram gives us 
        \[\int_MY(\phi)_nk_{T(i)}(m,n)f(m)d\mu(i)_n(m)=\int_Mk_{T(j)}(m,n)X(\phi)_mf(m)d\mu(j)_n(m)\]
        for all $(m,n)\in M\times N$ and all families $f(m)\in \int^{\oplus,p} X_m$. By the uniqueness of the integral kernel for this bounded operator $\int^{\oplus,p}_M X_m(i)\to \int^{\oplus,q}_N Y_n(j)$, we see that $\mu(i)$ and $\mu(j)$ are in the same measure class, and thus there is a measurable function $\mu(\phi):M\times N\to \R^{\geq 0}$ such that $d\mu(i)/d\mu(j)=\mu(\phi)$ and 
        \[Y(\phi)_nk_{T(i)}(m,n)=\mu(\phi)(m,n)k_{T(j)}(m,n)X(\phi)_m.\]
        Thus, by replacing  each $d\mu_{T_i}$ with $\frac{d\mu_{T(i)}}{d\mu_{T}}d\mu_{T}$ for some fixed measure $\mu_{T}$ in the common measure class and incorporating the Radon-Nikodym derivative into the kernel $k_{T(i)}$, we obtain a family of natural transformations $k_{T}(m,n):X_m\to Y_n$, in $\Cc$, such that 
        \[T=\int_Mk_{T}(m,n) (\_) d\mu_n(m)\]
        (where we can calculate all integrals level-wise).
    \end{proof}
\begin{corollary}{\label{Corol: Diag}}
    Assume the base field $\KK$ is Archimedean.
    Suppose $X=\int^{\oplus,p}_MX_m$ and $Y=\int^{\oplus,q}_NY_n$ be two objects in $\cat{C}$, and $f:X\to Y$ a morphism between them. 
    \begin{enumerate}
        \item If $\Hom_{\cat{C}}(X_m,Y_n)=0$ for all $m\in M, n\in N$, then $\Hom(\int^{\oplus}_MX_m,\int^{\oplus}_NY_n)=0$. 
        \item If $M=N$, and $\Hom_{\cat{C}}(X_m,Y_n)=0$ for all $m\neq n$, then $f$ is diagonalisable (in the image of $\int^{\oplus}$).
    \end{enumerate}
\end{corollary}

\begin{proof}
    \begin{enumerate}
    \item Let $f:\int^\oplus_M X_m\to\int^\oplus_NY_n$ be a $\cat{C}$-morphism. Then $f$ arises from a matrix $f(m,n)\in \Hom(X_m,Y_n)$ so that 
    \[f(x_m)_{m\in M}=\left(\int_Mf(m,n)x_m\right)_{n\in N}\]
    We then conclude that $f(m,n)=0$ for all $m,n$ by the hypothesis and obtain that $f=0$.

        \item We use Corollary \ref{Corol: integral kernels} to conclude that $f$ takes the form 
        \[f(i):(x_m)_{m\in M}\mapsto \left({\small\int}_M k_{f(i)}(m,n)x_m d\mu_n(m)\right)_{n\in N}\]
        and $k_{f(i)}(m,n)=0$ unless $m=n$. Thus, the only measures on $M$ that give anything non-trivial are the (multiples of) counting measure. Thus, 
        $f(i)(x_m)=k_{f(i)}(m,m)(x_m)$ and $f$ is diagonalisable.
        
    \end{enumerate}
\end{proof}

\section{Direct Integrals of Sheaves}

In this section, we return again to the general situation where the base field $\KK$ can be either Archimedean (with $p\in [1,\infty]$) or non-Archimedean (with $p=0 $ or $\infty$).
Up until this point in the paper we have used $x$ to denote an element of a Polish space $M$. Since we need a Polish space $M$ and a Topological space $X$, we use the letters $m$ and $x$ to denote elements of $M$ and $X$ respectively. Thus, to avoid unnecessary notational confusion will use $\int^{\textrm{meas}}$ instead of $\int^m$ to denote the measurable sections. 

\subsection{Conditions for detecting sheaves of measurable abstract Banach bundles}

    We now move onto the notion of sheaves of abstract Banach bundles.
    Throughout, we let $X$ be a second countable topological space.  

    Let $\cat{mPreSh}^M(X)$ be the category of functors $\Fc:\cat{Top}(X)^\textrm{op}\to\cat{mBan}^M$. As $\cat{mBan}^M$ is an elementary Quasi-abelian category (see Definition 2.1.10 of \cite{Schneiders1999}), one can define the category of sheaves valued in $\cat{mBan}^M$ just as in [\cite{Schneiders1999} \S 2.2.1], which we will denote by $\cat{mSh}^M(X)$. 
    The compositions of $\Fc$ with either the measurable sections functor $\int^{\textrm{meas}}$ or the evaluation functor $\{\ev_m\}_{m\in M}$ automatically provide us with two functors, 
    \[\cat{mPreSh}^M(X)\xrightarrow{\int^{\textrm{meas}},\ev_m}\cat{PreSh}(X)\]
    which are easily seen to be strictly exact as a consequence of Proposition \ref{Prop: Measurable exactness}.

\begin{prop}{\label{Prop: measurable sheafiness}}
    Let $\Fc\in\cat{mPreSh}^M(X)$. 
\begin{enumerate}
    \item  The presheaf $\Fc$ is a sheaf if and only if every $\ev_m\Fc$ is a sheaf.
    \item The presheaf $\int^{\textrm{meas}}\Fc$ is a sheaf if and only if $\ev_m\Fc$ is a sheaf for almost all $m$. 
\end{enumerate}
\end{prop}
  \begin{proof}
  View the abstract Banach bundles $\Fc(U)$ as belonging to the cocompletion of $\cat{mBan}^M$. By Proposition 2.1.15 of \cite{Schneiders1999}, this category is also complete. Therefore, because the sheaf property is equivalent to the strict exactness of the sequence 
   \begin{equation}\label{Eq: Sheaf Exact}
        0\to\Fc(U)\to\prod_i\Fc(U_i)\to\prod_{i,j}\Fc(U_i\cap U_j),
    \end{equation}
     the Lemma follows directly from Proposition \ref{Prop: Reflecting Exactness}.
  \end{proof}

We now provide conditions for a sheaf of abstract Banach bundles $(\Fc_m)_{m\in M}$ on a space $X$ to define an element of $\cat{mPreSh}^M(X)$, so that 
$$\int^{\textrm{meas}}\Fc_m(U)dm$$
is well defined.

\begin{definition}
    Let $\Fc=\{\Fc_m\}_{m\in M}$ be a sheaf of measurable abstract Banach bundles in $\cat{mSh}^M(X)$. We say that $\Fc$ is a locally constant (resp. constructible/skyscraper) sheaf of abstract Banach bundles if (almost) all $\Fc_m=\ev_m(\Fc)$ are themselves locally constant (resp. constructible/skyscraper) sheaves. 
\end{definition}

    We consider now the category of sheaves of bounded abstract Banach bundles $\cat{bSh}^M(X)$, defined in the same manner as $\cat{mSh}^M(X)$. A priori, it seems intuitive that there should be no difference between the direct integral and the measurable sections functor, but this would be incorrect, and as even this simple example shows, the direct integral of a sheaf may not output a sheaf in general.

\begin{counterexample}

    Let $M=X=\C$, and let $\Fc=\{\Fc_m\}$ be the sheaf of abstract Banach bundles with $\Fc_m$ the skyscraper sheaf at $m\in X$.
        Then we can calculate the direct integral as
        \[\left(\int^{\oplus,p}_M\Fc_mdm\right)(U)=\int^{\oplus,p}_M\Fc_m(U)dm=L^p(U).\]
        This defines a presheaf. However, because there is a strictly countable cover $\bigcup_{i\in \N} U_i$ of $X=\C$ such that each $U_i$ has finite measure, it follows that the constant function $i\in\KK\subseteq L^p(U_i)$ but, when we glue to get a function on $X$, the constant function on $X$ is not in $L^{p}$. 
\end{counterexample}

We explain this phenomenon with help of the following Proposition. 

\begin{prop}{\label{prop: notsheaf}}
    Let $\Fc\in \cat{bPreSh}^M(X)$ be a presheaf $\Fc:\cat{Top}(X)^\textrm{op}\to \cat{bBan}^M$. 
\begin{enumerate}
    \item  The presheaf $\Fc$ is a sheaf if and only if every $\ev_m\Fc$ is a sheaf.
    \item If almost all $\ev_m\Fc$ are sheaves, then the presheaf $\int^{\oplus,p}\Fc$ is local, and satisfies \emph{finite} gluing. 
\end{enumerate}
\end{prop}

\begin{proof}
\begin{enumerate}
    \item  The proof of the first item is similar to that of $\cat{mBan}^M$, the point being that one can view the outputs as objects in the cocompletion, and the functor $\ev_m$ can be uniquely extended to this cocompletion and remain strongly exact, the rest of the argument of Lemma \ref{Prop: measurable sheafiness} works out. 
    \item For the second item, consider an open set $U$ and a \emph{finite} cover $\{U_i\}_{i\in I}$. Then finite gluing and finite locality follows from the strict exactness of the sequence
    \begin{equation}\label{Eq: Sheaf Exact 2 }
        0\to\int^{\oplus,p}\Fc(U)\to\prod_{i\in I}\int^{\oplus,p}\Fc(U_i)\to\prod_{i,j\in I}\int^{\oplus,p}\Fc(U_i\cap U_j)
    \end{equation}
    which makes sense in $\cat{bBan}^M$ because each product $\prod_{i\in I}=\bigoplus_{i\in I}$ when $I$ is finite. 

    The strict exactness of this sequence follows directly from the strong exactness of the direct integral $\int^{\oplus,p}$ (see Proposition \ref{Prop: DI strong exact}).

    To show that the locality axiom holds for arbitrary covers, it follows simply because, when viewed as sheaves on vector spaces, $\int^{\oplus,p}\Fc_m$ is a subsheaf of $\int^{\textrm{meas}}\Fc_m$, and Proposition \ref{Prop: measurable sheafiness} tells us that $\int^{\textrm{meas}}\Fc_m$ is a sheaf.
\end{enumerate}  
\end{proof}

\begin{remark}
    Because of the example preceding the last proposition, the direct integral of sheaves of abstract Banach bundles need not generally satisfy the infinite gluing axiom. The reason for this discrepancy is that the measurable sections functor acts like a product (see Example \ref{Ex: measurable sections sim product}), and the direct integral, like a coproduct, so that $\int^{\textrm{meas}}\prod_I=\prod_I\int^{\textrm{meas}}$, but $\int^{\oplus,p}\prod_I\neq\prod_I\int^{\oplus,p}$.
\end{remark}
        Given a sheaf of abstract Banach bundles $\Fc:\cat{Top}(X)^{\textrm{op}}\rightarrow \cat{bBan}^M$, when we refer to the direct integral, we will mean the sheafification of the presheaf obtained by post-composing with the functor $\int^{\oplus,p}_M:\cat{bBan}^M\rightarrow \cat{Vect}$.  
        In our example, we calculate the direct integral as
        \[\int^{\oplus,p}_M\Fc_mdm=\Lc^p \textrm{ where } \Lc^p(U)=L^p_{\loc}(U).\]
        
\subsection{Some criteria for sheaves valued in $\cat{bBan}^M$}

\begin{prop} Let $M$ be a measure space, and assume that $X$ is a second countable Hausdorff topological space, with Borel $\sigma$-algebra.
    \begin{enumerate}
        \item 
        Let $(\Fc_m)$ be a family of locally constant sheaves on $X$, which we can identify with a family of representations of $\pi_1(X)$. If the map
        \begin{align*}
            M&\to \coprod_n \Hom(\pi_1(X),\GL_n(\C) )\\
            m&\mapsto \Fc_m
        \end{align*}
        is measurable, then there is a sheaf 
        $\Fc:\cat{Top}(X)^{\operatorname{op}}\to \cat{mBan}^M$
        such that each $\Fc_m=\ev_m\Fc$.
        \item Let $(\Fc_m)$ be a sheaf of abstract Banach bundles on $X$ such that $\Fc_{m}$ is a skyscraper sheaf at the point $x_m\in X$. 
        Suppose the maps $P:m\in M\mapsto x_m\in X$ and $d:m\in M\mapsto \dim(\Fc_{m,x_m})\in \N$ are measurable. Then there is a sheaf 
        $\Fc:\cat{Top}(X)^{\operatorname{op}}\to \cat{mBan}^M$
        such that each $\Fc_m=\ev_m\Fc$.
    \end{enumerate}
\end{prop}

\begin{proof}
 It is enough to show that both in cases we can construct a functor (i.e. a presheaf) $\Fc:\cat{Top}(X)^{\operatorname{op}}\to \cat{mBan}^M$ such that $\Fc_m=\ev_m\Fc$ for every $m\in M$. The fact that these presheaves are in fact sheaves is the content of Proposition \ref{Prop: measurable sheafiness}.
    \begin{enumerate}

        \item First of all, we want to show that for any $m\in M$ and $U\subseteq X$ we can endow $\Fc_m(U)$ with a Markushevic basis so that $(\Fc_m(U),\{b_i(m)\}_{i\in\N})_{m\in M}\in\cat{mBan}^M$. For every $m\in M$, we denote by $\rho_m\in \operatorname{Hom}(\pi_1(X),\operatorname{GL}_n(\C))$ the finite dimensional representation of $\pi_1(X)$ associated to $\Fc_m$. Choose a point $x\in X$, and let $V_m=\Fc_{m,x}$ be the stalk at $x$. We may assume, by splitting $M=\bigcup_n M_n$ if necessary, that the dimension of $V_m$ is constant, so that $V_m=\C^n$ for some fixed $n\in\N$ and we endow each $V_m$ with the standard orthonormal basis of $\C^n$ (which is of course a Markushevich basis). Now, one can choose a simply connected space $\tilde{X}$ such that $p:\tilde{X}\rightarrow X$ is a bijective continuous map, and is locally a homeomorphism (for example, one can take a fundamental domain in a universal cover).

        As $\tilde{X}$ is simply connected, for any point $y\in \tilde{X}$ there is a unique path (up to homotopy) $y$ to $x$, and we can use this path to canonically identify the stalks $\gamma_y:\Fc_{m,y}\xrightarrow{\sim} \Fc_{m,x}$. Now, let $U\subseteq X$ be an open subset. For each $y\in U$, we obtain an injective map $\Fc_m(U)\to\Fc_{m,y}\xrightarrow{\gamma_y}\Fc_{m,x}$, where we can identify 
        \begin{align*}
            \Fc_m(U)=(\Fc_{m,y})^{\im(\pi_1(U)\rightarrow \pi_1(X))}&=\bigcap_{T_{i}}\ker(\rho_m(T_{i})-1) \\
            &=\ker\bigg(\{\rho_m(T_i)-I\}:\Fc_{m,y}\to \Fc_{m,y}^{\#\{T_i\} }\bigg)
        \end{align*} where $T_i$ is a finite set of generators of  $\im(\pi_1(U)\rightarrow \pi_1(X))$. By the assumptions, the maps $m\mapsto \rho_m(T_i)$ are all measurable, so that $(\Fc_m(U))$ becomes a measurable abstract Banach bundle, inheriting a Markushevich basis from the kernel construction of Lemma \ref{Lem: ker and coker}. This shows that $(\Fc_m(U),\{b_i(m)\}_{i\in\N})_{m\in M}\in\cat{mBan}^M$.
        
        To see that the association $U\mapsto(\Fc_m(U),\{b_i(m)\}_{i\in\N})_{m\in M}$ defines a functor $\Fc:\cat{Top}(X)^{\operatorname{op}}\to \cat{mBan}^M$ we need to see that 
        $\Fc(\operatorname{Res}:U\to V)$ is a measurable morphism. 
        
        We have \begin{align*}
            \Fc_m(U)&=(\Fc_{m,y})^{\im(\pi_1(U)\to \pi_1(X))}\\
            &=\big((\Fc_{m,y})^{\im(\pi_1(V)\to \pi_1(X))}\big)^{\im(\pi_1(U)\to \pi_1(X))} \\
            &=\Fc_m(V)^{\im(\pi_1(U)\to \pi_1(X))}
        \end{align*}
        so the restriction map 
        $\Fc_m(U)\to \Fc_m(V)$ is identified with the inclusion map of those vectors of $\Fc_m(V)$ that are fixed by all elements of $\im(\pi_1(U)\to \pi_1(X))$ into $\Fc_m(V)$. Since the fixed vectors, as shown earlier, are the intersection of the kernels of $\{\rho_m(T)-I\}$ where $T$ varies over a finite set of generators of this group, it follows that this is measurable.

        \item For $m\in M$, we set $d(m):=\operatorname{dim}(\Fc_{m,x_m}),$ so that for $U\subset X$ open 
        $$\Fc_m(U)=\begin{cases}
            \C^{d(m)},\quad \text{if } x_m\in U;\\
            0,\qquad\quad\ \text{otherwise}.
        \end{cases}$$

We define the association 
\begin{align*}
    \Fc:\cat{Top}(X)^{\operatorname{op}}&\to\cat{mBan}^M \\
    U&\mapsto\Big(\Fc_m(U),\{b_i(m)\}_{i\in\N}\Big)_{m\in M}
\end{align*}
where $b_i(m):=\begin{cases}
    e_i,\quad &x_m\in  U,\ i\le d(m)\\
    0, & \text{otherwise}.
\end{cases}$

We need to show that $\Fc_m(U)\in \cat{mBan}^M$. If $x_m\notin U$, then $\Fc_m(U)$ is the trivial abstract Banach bundle in $\cat{mBan}^M$. If $x_m\in U$, we need to make sure that, for any sequence $(a_i)$ in $\C$ such that all but finitely many $a_i$'s are zero, the function 
\begin{align*}
    X&\to\C \\
    m&\mapsto\Big|\Big|\sum_i a_i b_i(m)\Big|\Big|_{m}
\end{align*}
is measurable. Notice that 
$$\Big|\Big|\sum_i a_i b_i(m)\Big|\Big|_{m}=\Big|\Big|\sum_i a_i b_i(m)\Big|\Big|_{\C^{d(m)}}=\sqrt{\sum_{i=1}^{d(m)}|a_i|^2}$$
and the latter expression is measurable on account of $m\mapsto d(m)$ being measurable. Now, in order to show that $\Fc$ is a functor, we need to show that if $V\subseteq U$, then $\Fc(\operatorname{Res}:U\to V)=:\phi$ is a morphism in $\cat{mBan}^M$. Observe that $\phi$ is defined pointwise by 
$$\phi_m:=\begin{cases}
     \operatorname{Id}:\C^{d(m)}\to\C^{d(m)},\qquad &x_m\in V\\
     0, & x_m\not\in V.
\end{cases}$$

For $i,j\in \N$ the matrix coefficients of $\phi$ are
$$\phi_{i,j}(m)=\begin{cases}
    \delta_{i,j}\quad & x_m\in V\\0, & x_m\not\in V.
\end{cases}$$
In particular, $\phi$ is a step morphism (cf. Def. \ref{Def: morphisms}) with the decomposition 
\[M=\bigcup_{n=1}^{\infty}\big((P^{-1}(V)\cap d^{-1}(n)\big)\sqcup \big(P^{-1}(V)^c)\cap d^{-1}(n)\big).\]
    \end{enumerate}
\end{proof}

\subsubsection{Examples}

\begin{example}
    Let $M=\R,\ X=\R^2$ and $L=\C$. Denote by $\pi_X:\R^2\to\R$ the projection onto the first component. We define a functor $\mathcal{F}\in\cat{Top}(X)^{\operatorname{op}}$ by 
    $\mathcal{F}(U)=\{F_m(U)\}_{x\in\R}$, where 
    $$F_m(U)=\begin{cases}
        \C\quad\text{if } (m,0)\in U,\\
        0\quad\ \text{ otherwise}.
    \end{cases}$$
    Applying the measurable section functor yields 
    $$\int_\R^{\textrm{meas}} F_m(U) dm=\{ f:\pi_X(U)\to\C\ \text{measurable}\}.$$
    For $p\in[1,\infty]$, the direct integral of $L^p$-sections becomes 
    $$\int_R^{\oplus,p} F_m(U) dm=L^p_{\operatorname{loc}}(\pi_X(U)).$$
\end{example}
\begin{example}[Integral of skyscraper sheaves]
    Applying the measurable section functor yields 
    $$\int_M^{\textrm{meas}} \Fc_m(U) d\mu(m)=\{ f:\varphi^{-1}(U)\to \C\ \text{measurable}\}.$$
    For $p\in[1,\infty]$, the direct integral of $L^p$-sections becomes (after sheafification)
    $$\int_M^{\oplus,p} \Fc_m(U)d\mu(m)=L^p_{\operatorname{loc}}(\varphi^{-1}(U)).$$
\end{example}

We set up some notation for the next example.
Let $M$ be a Borel space, and let $L^2(M)$ be the set of $L^2$ functions on $M$. For a point $m\in M$, we let \[L_{\lim}^2(m):=\varinjlim_{\substack{ m\in U  \\ U\textrm{ open}}} L^2(U) \] be the equivalence classes of $L^2$ functions where $f\sim g$ if and only if there is a closed set $Z\not \ni m$ such that $f|_{M\setminus Z}=g|_{M\setminus Z}$ almost everywhere. If $M$ is not compact, we similarly let $L_{\lim}^2(\infty)$ be the equivalence classes of $L^2$-functions where two functions $f,g$ are equivalent if there is a compact subset $Z$ such that $f|_{M\setminus Z}=g|_{M\setminus Z}$.

More generally, whenever $S\subseteq M$ is some subset of $M$, we can define 
\[L_{\lim}^2(S):=\varinjlim_{\substack{ S\subseteq  U  \\ U\textrm{ open}}} L^2(U) \]

\begin{example}
     Let $M=\C$, $X=\Abb^2$. For each $m\in M$, let $\Fc_m$ be the rank 1 sheaf supported on $y=mx^2$. Then the sheaf 

\[\Fc^{\int}=\int^{\oplus}\Fc_mdm\]
has the following stalks: 
\[\Fc^{\int}_x=\begin{cases}L_{\lim}^2(y/x) & \textrm{if } x\neq 0 \\
L_{\lim}^2(\infty)^2 & \textrm{if } x=0,y\neq 0 \\
L^2(\C) & \textrm{if } (x,y)= 0
\end{cases}\]
\end{example}
    
\begin{example}
     Let $X=M=\C$, let $j_m:\C\setminus \{m\}\rightarrow \C$, and let  $\Fc_m=(j_m)_!\underline{\C}_{\C\setminus\{m\}}$ be the zero extension of the constant sheaf. Then 
\[\int^\oplus\Fc_m dm (U)= L^2[\C\setminus U] \]
and the stalk at $m\in X=\C$ is the space of $L^2$ functions that are identically a.e. zero in some neighbourhood of $m$. That is, the local functions in $\int^\oplus \Fc_m$ are precisely the functions that are zero near $m$.
\end{example}

\appendix
\section{Appendix: Competing notions of measurability}{\label{App: A}}

One of the technicalities in this work was choosing a suitable notion of measurable families of operators that works well with the algebraic properties we want. To do this, there are multiple choices one can make, as is already clear in the case where the families of Banach spaces $(B_x)$ are constant (say, each $B_x=B$). 
Here, families $(v_x)\in (B_x)$ coincide with functions $M\to B$ for which it is well known there are multiple notions of measurability, of which we list some important ones below. Suppose $f:M \to B$ is a function. Then:
\begin{description}
    \item[$f$ is norm measurable] if the preimage $f^{-1}(U)\subseteq M$ is measurable for every $U\subseteq B$ in the $\sigma$-algebra arising from the norm-topology;
    \item[$f$ is Bochner measurable] (also known as strongly measurable) if $f$ is a pointwise limit of step functions $f_n:M\to B$;
    \item[$f$ is weakly measurable] if the preimage $f^{-1}(U)\subseteq M$ is measurable for every $U\subseteq B$ in the $\sigma$-algebra arising from the weak topology.
\end{description}

The following facts are well known:

\begin{itemize}
    \item Bochner measurable $\implies $ Weakly measurable.
    \item If $B$ is separable, then Weakly measurable $\implies $ Norm measurable $\implies $ Bochner measurable.
    \item (\textbf{Pettis measurability theorem}) If $f$ is weakly measurable and the image $f(M)\subseteq B$ is (almost surely) separable, then $f$ is Bochner measurable.  
\end{itemize}

In our generalised context, there are generalisations of Bochner and weak measurability. Our notion of morphism in $\cat{mBan}^M_\omega$ (see Definition \ref{Def: morphisms}) agrees readily with the notion of Bochner measurability, and our notion of measurable sections (see Definition \ref{Def: measurable sections}) agrees readily with weak measurability, but it remains to be seen whether they are equivalent in appropriate circumstances. 

\begin{theorem}{\label{Thm: Separable measurablity}}
    Let $(B_x)$ be an abstract Banach bundle in $\cat{mBan}_{\omega}^M$, with $\omega=\N$. Then $(v_x)$ is a weakly measurable section if and only if it is Bochner measurable (that is, it is the pointwise limit of step families $(v_x^n)\in (B_x)$). 
\end{theorem}

\begin{proof}
That Bochner measurable implies weakly measurable is a special case of Lemma \ref{matrix coefficients are measurable maps}. 

    Let $w^N_x:=\sum_{i\leq N} b^*_i(v_x)b_i$. Each $w^N_x$ is weakly measurable as a function $M\to k^N$, and is therefore Bochner measurable, although it is not guaranteed to converge to $v_x$ as $N\to \infty$. 

    Let $\delta_N>0$ be some sequence of numbers that tends to zero, and let $L^N\subseteq k^N$ be a lattice, fine enough so that for every point $x\in k^N$, there is some lattice point $p\in L^N$ such that $\|x-p\|_{\ell^1}<\delta_N$.
    
    Consider the family $w^N_x-p-v_x$, with $p\in L^N$. Each family is weakly measurable by hypothesis, so the norm function 
    \[x\mapsto \|w_x^N-p-v_x\|_x=\sup_{\substack{\phi\in \Span_{\Q}(b_i^*(x))\\\|\phi\|\leq 1}}|\phi(w_x^N-p-v_x)|\] 
    is also measurable, being the supremum of countably many measurable functions. 
    Hence, it is clear that the function
    \begin{equation*}
        x\mapsto \min_{p\in L^N}\{\|w^N_x-p-v_x\|_x\}
    \end{equation*}
    is also measurable. 

    Put some order on the elements of $L^N=\{p_i\}_{i\in \N}$, define \[M'_i=\{x\in M: \textrm{The minimum }\min_{p\in L^N}\{\|w^N_x-p-v_x\|_x\} \textrm{ is attained at }p_i\},\]
    and set $M_i:=M'_i\setminus\bigcup_{j<i}M'_j$.

     Since $w_x^N-p_i$ is measurable as a function $M_i\to k^N$, it can be approximated to arbitrary degree by a step family. 
     Let $s^{N,i}_x$ be a step family on $M_i$ such that $\|s^{N,i}_x-(w_x^N-p_i)\|_x<\delta_N$. 
     For $x\in M_i$, we have the inequalities
\begin{align*}
    \|s^{N,i}_x-v_x\|_x &\leq \|s_x^{N,i}-(w^N_x-p)\|_x + \|w^N-p-\mu_N(v_x)\|_x+\|\mu_N(v_x)-v_x\|_x \\
    &<\delta_N+\|w^N-p-\mu_N(v_x)\|_{\ell^1}+\|\mu_N(v_x)-v_x\|_x \\
    &<2\delta_N +\operatorname{dist}(B^N_x,v_x).
\end{align*}
Combining these $(s_x^{N,i})_{x\in M_i}$ into one family $(s_x^N)_{x\in M}$, we have a step function such that for each $x\in M$, 
\[\|s^{N}_x-v_x\|_x<2\delta_N+\operatorname{dist}(B^N_x,v_x).\]
As, for fixed $x$, both quantities on the right vanish as $N\to\infty$, it follows that $(v_x)$ is the pointwise limit of a sequence of step families.  
\end{proof}

\section{Appendix: Archimedean and non-Archimedean functional analysis}{\label{App: B}}

In functional analysis, there are many deep analytical and geometric differences between the Archimedean and non-Archimedean cases, with surprising consequences. 
In this paper, we frequently rely on fundamental and deep results of functional analysis, which are well known in the Archimedean case and a little less well known in the non-Archimedean case. To this end, we list these required classical  theorems below, provide references for their non-Archimedean analogues, and, when necessary, point out the important additional conditions for the result to hold in the non-Archimedean case.

\begin{description}
    \item[Open mapping theorem] Proposition 8.6 \cite{Sch02}
    \item[Hahn--Banach] Proposition 9.2 \cite{Sch02}.  The field $\KK$ must be spherically complete. 
    \item[Banach--Alaoglu] 
    $V$ a Banach space and $V^\vee$ the dual space. Then
    the closed unit ball $B_{\leq 1}\subseteq V^\vee$ is compact in the weak-$^*$ topology.

    In the non-Archimedean setting, this result is Lemma 13.1(vi) of \cite{Sch02}, provided the non-Archimedean field $\KK$ is locally compact. (This is simply to force the notion of compactness to coincide with c-compactness).
    \item[Existence of Markushevich bases] 
    Every separable Archimedean Banach space has a Markushevich basis \cite{fabian2011banach}.
    Every separable non-Archimedean Banach space not only has an M-basis, it also has a Schauder basis and is consequently linearly homeomorphic to $c_0(\N)$.
    \item[Radon-Nikodym] We quote this result in a section where we only consider the Archimedean case, but it does have a non-Archimedean version, which one can find in Section 2, Theorem 4 of \cite{MR1874423}.
\end{description}

\bibliographystyle{alpha}
\bibliography{ref}
\end{document}